\documentclass[a4paper,10pt]{article}
\usepackage[utf8]{inputenc}
\usepackage{mathrsfs,pgf,tikz,amsmath,amssymb,enumitem,hyperref,amsthm}
\usetikzlibrary{arrows}

\theoremstyle{plain}
\newtheorem{theorem}{Theorem}[section]                                          
\newtheorem{proposition}[theorem]{Proposition}                          
\newtheorem{lemma}[theorem]{Lemma}

\theoremstyle{definition}

\theoremstyle{remark}
\newtheorem{remark}[theorem]{Remark}

\DeclareMathOperator{\p}{P}

\title{Complexities of Erez self-dual normal bases}
\author{Blondeau Da Silva Stéphane}

\begin{document}

\maketitle

\begin{abstract}
The complexities of self-dual normal bases, which are candidates for the lowest complexity basis of some defined extensions, are determined with the 
help of the number of all but the simple points in well chosen minimal Besicovitch arrangements. In this article, these values are first compared with 
the expected value of the number of all but the simple points in a minimal randomly selected Besicovitch arrangement in ${\mathbb{F}_d}^2$ for the 
first $370$ prime numbers $d$. Then, particular minimal Besicovitch arrangements which share several geometrical properties with the arrangements 
considered to determine the complexity will be considered in two distinct cases.
\end{abstract}

\section*{Introduction}

Let $q$ be a prime power, $\mathbb{F}_q$ be the field of $q$ elements and $n$ be a positive integer. We consider the Galois group of the 
extension $\mathbb F_{q^n}/\mathbb F_q$, which is a cyclic group generated by the Frobenius automorphism $\Phi:x\mapsto x^q$. There exists an $\alpha$ 
that generates a "normal" basis for $\mathbb F_{q^n}/\mathbb F_q$, \textit{i.e.} a basis consisting of the orbit $(\alpha,\alpha^q,...,
\alpha^{q^{n-1}})$ of $\alpha$ under the action of the Frobenius. The difficulty of multiplying two elements of the extension expressed in this basis is 
measured by the complexity of $\alpha$, namely the number of non-zero entries in the multiplication-by-$\alpha$ matrix: $\big(Tr(\alpha\alpha^{q^i}
\alpha^{q^j})\big)_{0\le i,j\le n-1}$, where $Tr$ is the trace map from $\mathbb F_{q^n}$ to $\mathbb F_q$ (\cite[4.1]{Menezes}). As a large number of 
zero in this matrix enables faster calculations, finding normal bases with low complexity is a significant issue. 

Self-dual normal bases are particular normal bases which verify $Tr(\alpha^{q^i}\alpha^{q^j})=\delta_{i,j}$ (for $0\le i,j\le n-1$), where $\delta$ is 
the Kronecker delta. Arnault \textit{et al.} in \cite{APV} have identified the lowest complexity of self-dual normal bases for extensions of low degree 
and have showed that the best complexity of normal bases is often achieved from a self-dual normal basis. In \cite{PV}, Pickett and Vinatier considered 
cyclotomic extensions of the rationals generated by $d^2$-th roots of unity, where $d$ is a prime. The construction they use yields a candidate for the 
lowest complexity basis for $\mathbb F_{p^d}/\mathbb F_p$, where $p\neq d$ is a prime which does not split in the chosen extension. They prove that the 
multiplication table of this basis can be geometrically interpreted by means of an appropriate minimal Besicovitch arrangement. The complexity of 
the basis, denoted by $C_d$, is here equal to the number of all but the simple points generated by this arrangement in ${\mathbb{F}_d}^2$.

After a brief overview of the properties this arrangement have, we will compare the complexity $C_d$ with the expected value of the number of all but the 
simple points in a minimal randomly selected Besicovitch arrangement in ${\mathbb{F}_d}^2$ for the first $370$ prime numbers $d$. The expectations will 
be determined using Blondeau Da Silva's results in \cite{ste}. In a third part, we will consider particular minimal Besicovitch arrangements which share 
several geometrical properties with the arrangements considered to determine the complexity. We will again compare in this part, for the first $370$ 
prime numbers $d$, $C_d$ with the expected value of the number of all but the simple points in the randomly selected mentioned above arrangement.

\section{The minimal Besicovitch arrangement providing the complexity}\label{s1}

Let $d$ be a prime number and ${\mathbb{F}_d}$ be the $d$ elements finite field.

A line, in ${\mathbb{F}_d}^2$, is a one-dimensional affine subspace. A Besicovitch arrangement $B$ is a set of lines that contains at least one line in 
each direction. A minimal Besicovitch arrangement is a Besicovitch arrangement that is the union of exactly $d+1$ lines in ${\mathbb{F}_d}^2$ (see 
\cite{ste}).

The minimal Besicovitch arrangement considered, brought out by Pickett and Vinatier (\cite{PV}), and denoted by $\mathscr{L}$, is composed of $d+1$ 
lines with the following equations:
\[\left \{\begin{array}{l @{} l}
    L_a:& \quad ax-(a+1)y-p(a)=0\qquad  \text{for }a\in{\mathbb{F}_d}, \\
    L_\infty:& \quad x-y=0\enspace,
\end{array}
\right.\]
where $p$ is the following polynomial:
\begin{align}\label{equ1}
\forall x\in\mathbb{F}_d,\quad p(x)=\frac{(x+1)^d-x^d-1}{d}.
\end{align}

For $d\ge5$, Pickett and Vinatier (\cite{PV}) have proved that under the action of $\Gamma=\langle\iota,\theta\rangle$ (a group generated by two 
elements of $GL_2(\mathbb{F}_d)$, where $\iota(x,y)=(y,x)$ and $\theta(x,y)=(y-x,-x)$ for $(x,y)\in{\mathbb{F}_d}^2$), this arrangement $\mathscr{L}$ 
always has two orbits of cardinal $3$: $\{L_0,L_{-1},L_\infty\}$ and $\{L_1,L_{\frac{d-1}{2}},L_{-2}\}$. They have also stated that:
\begin{enumerate}[label=$\bullet$]
 \item if $d \equiv 1 \mod 3$, there are one orbit of cardinal $2$, $\{L_\omega,L_{\omega^2}\}$, where $\omega$ is a primitive cubic root of unity in 
 $\mathbb{F}_d$ and $\frac{d-7}{6}$ orbits of cardinal $6$;
 \item if $d \equiv 2 \mod 3$, there are $\frac{d-5}{6}$ orbits of cardinal $6$.
\end{enumerate}

The Comp\_lib $1.1$ package have been implemented in Python $3.4$. It provides the complexity $C_d$ of the basis (by counting all but the simple points 
in the associated minimal Besicovitch arrangement) and it also enables to determine the points multiplicities distribution in ${\mathbb{F}_d}^2$ of this 
arrangement. It is available at \url{https://pypi.python.org/pypi/Comp_lib/1.1}. Table \ref{Comp} in Appendix gathers the first $370$ values of $C_d$.

\section{Complexity \textit{versus} number of all but simple points in randomly selected arrangements}

Let us denote by $A_d$ the expected value of the number of all but the simple points in a randomly chosen minimal Besicovitch arrangement in 
${\mathbb{F}_d}^2$. Thanks to the proof of Theorem 1. in \cite{ste}, we have:
\begin{align*}
A_d&=d^2-d(d+1)(1-\frac{1}{d})^d\\
&=(1-\frac{1}{e})d^2-\frac{1}{2e}d+O(1),\quad\text{as $d\rightarrow\infty$}.
\end{align*}

Figure \ref{Fig} shows the values of $\frac{C_d-A_d}{d}$ for the first $370$ prime numbers.

\begin{figure}[ht]
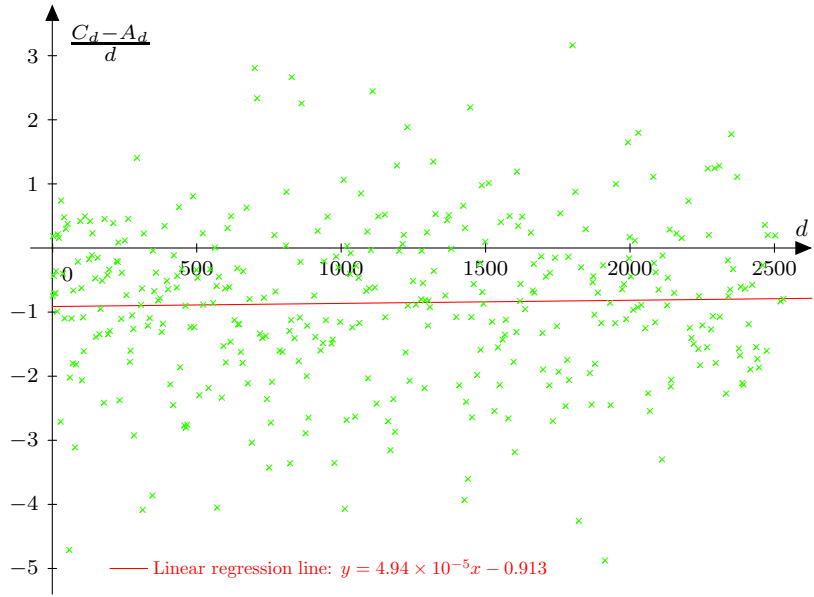

\centering
\definecolor{ffqqqq}{rgb}{1.,0.,0.}
\definecolor{ttffqq}{rgb}{0.2,1.,0.}

\caption{The $370$ values of the function that relates each prime number $d$ to $\frac{C_d-A_d}{d}$.}
\label{Fig}
\end{figure}

\subsection{A first test}

From the $370$ values of Figure \ref{Fig}, we plot the regression line: its slope $s$ is approximately $4.94\times10^{-5}$ and its intercept is 
approximately $-0.913$. 

Let us consider the following null hypothesis $H_0$: $s=0$. We have to calculate $T=\frac{s-0}{\hat\sigma_s}$, where $\hat\sigma_s$ is the 
estimated standard deviation of the slope. We obtain $\hat\sigma_s\approx 8.74\times10^{-5}$ and $T\approx 0.565$. $T$ follows a student's 
t-distribution with $(370-2)$ degrees of freedom (see \cite[Proposition 1.8]{cor}). The acceptance region of the hypothesis test with a $5\%$ risk is 
approximately $[-1.967,1.967]$. Thus it can be concluded that we cannot reject the null hypothesis : the fact that the slope is not 
significantly different from zero can not be rejected.

\subsection{A second test}

Figure \ref{Fig4} below shows the distribution of the values of $\frac{C_d-A_d}{d}$ for the first $370$ prime numbers. In regard to the resulting 
histogram, one may wonder whether these values are normally distributed or not.

\begin{figure}[ht]
\centering
\definecolor{qqffqq}{rgb}{0.,1.,0.}
\begin{tikzpicture}[line cap=round,line join=round,>=triangle 45,x=1cm,y=0.15cm]
\draw[->,color=black] (-5.26744744440467,0.) -- (3.6,0.);
\foreach \x in {-5,-4,-3,-2,-1,1,2,3}\draw[shift={(\x,0)},color=black] (0pt,2pt) -- (0pt,-2pt) node[below] {\footnotesize $\x$};
\draw[->,color=black] (0.,-1.2257230782525184) -- (0.,39.48401818184884);
\foreach \y in {5,10,15,20,25,30,35}\draw[shift={(0,\y)},color=black] (2pt,0pt) -- (-2pt,0pt) node[left] {\footnotesize $\y$};
\draw[color=black] (-3pt,-8.3pt) node {\footnotesize $0$};
\clip(-5.26744744440467,-1.2257230782525184) rectangle (3.6,39.48401818184884);
\fill[color=qqffqq,fill=qqffqq,fill opacity=0.5] (-5.,0.) -- (3.25,0.) -- (3.25,1.) -- (2.5,1.) -- (2.5,3.) -- (2.25,3.) -- (2.25,1.) -- (2.,1.) -- (2.,3.) -- (1.75,3.) -- (1.75,1.) -- (1.5,1.) -- (1.5,4.) -- (1.25,4.) -- (1.25,7.) -- (1.,7.) -- (1.,6.) -- (0.75,6.) -- (0.75,9.) -- (0.5,9.) -- (0.5,25.) -- (0.25,25.) -- (0.25,24.) -- (0.,24.) -- (0.,25.) -- (-0.25,25.) -- (-0.25,31.) -- (-0.5,31.) -- (-0.5,39.) -- (-0.75,39.) -- (-0.75,27.) -- (-1.,27.) -- (-1.,37.) -- (-1.25,37.) -- (-1.25,24.) -- (-1.5,24.) -- (-1.5,18.) -- (-2.,18.) -- (-2.,17.) -- (-2.25,17.) -- (-2.25,14.) -- (-2.5,14.) -- (-2.5,11.) -- (-2.75,11.) -- (-2.75,6.) -- (-3.,6.) -- (-3.,4.) -- (-3.5,4.) -- (-3.5,1.) -- (-3.75,1.) -- (-3.75,2.) -- (-4.,2.) -- (-4.,3.) -- (-4.25,3.) -- (-4.25,1.) -- (-5.,1.) -- cycle;
\draw (2.9,6) node[scale=1.2] {$\frac{C_d-A_d}{d}$};\draw (1.1,36) node {$Frequency$};\draw (3.,1.)-- (3.,0.);\draw (3.25,1.)-- (3.25,0.);
\draw (-0.75,39.)-- (-0.5,39.);\draw (-1.25,37.)-- (-1.,37.);\draw (-0.5,31.)-- (-0.25,31.);\draw (-1.,27.)-- (-0.75,27.);\draw (-0.25,25.)-- (0.,25.);
\draw (0.,24.)-- (0.25,24.);\draw (0.25,25.)-- (0.5,25.);\draw (-1.5,24.)-- (-1.25,24.);\draw (-2.,18.)-- (-1.5,18.);\draw (-2.,17.)-- (-2.25,17.);
\draw (-2.25,14.)-- (-2.5,14.);\draw (-2.5,11.)-- (-2.75,11.);\draw (-2.75,6.)-- (-3.,6.);\draw (-3.,4.)-- (-3.5,4.);\draw (-3.5,1.)-- (-3.75,1.);
\draw (-3.75,2.)-- (-4.,2.);\draw (-4.,3.)-- (-4.25,3.);\draw (-4.25,1.)-- (-5.,1.);\draw (3.25,1.)-- (2.5,1.);\draw (2.5,3.)-- (2.25,3.);
\draw (2.,3.)-- (1.75,3.);\draw (2.25,1.)-- (2.,1.);\draw (1.75,1.)-- (1.5,1.);\draw (1.5,4.)-- (1.25,4.);\draw (1.25,7.)-- (1.,7.);
\draw (1.,6.)-- (0.75,6.);\draw (0.75,9.)-- (0.5,9.);\draw (-5.,1.)-- (-5.,0.);\draw (-4.75,1.)-- (-4.75,0.);\draw (-4.5,1.)-- (-4.5,0.);
\draw (-4.25,3.)-- (-4.25,0.);\draw (-4.,3.)-- (-4.,0.);\draw (-3.75,2.)-- (-3.75,0.);\draw (-3.5,4.)-- (-3.5,0.);\draw (-3.25,4.)-- (-3.25,0.);
\draw (-3.,6.)-- (-3.,0.);\draw (-2.75,11.)-- (-2.75,0.);\draw (-2.5,14.)-- (-2.5,0.);\draw (-2.25,17.)-- (-2.25,0.);\draw (-2.,18.)-- (-2.,0.);
\draw (-1.75,18.)-- (-1.75,0.);\draw (-1.5,18.)-- (-1.5,24.);\draw (-1.5,18.)-- (-1.5,0.);\draw (-1.25,37.)-- (-1.25,0.);\draw (-1.,37.)-- (-1.,0.);
\draw (-0.75,39.)-- (-0.75,0.);\draw (-0.5,39.)-- (-0.5,0.);\draw (-0.25,31.)-- (-0.25,0.);\draw (0.,25.)-- (0.,0.);\draw (0.25,25.)-- (0.25,0.);
\draw (0.5,25.)-- (0.5,0.);\draw (0.75,9.)-- (0.75,0.);\draw (1.,7.)-- (1.,0.);\draw (1.25,7.)-- (1.25,0.);\draw (1.5,4.)-- (1.5,0.);
\draw (1.75,3.)-- (1.75,0.);\draw (2.,3.)-- (2.,0.);\draw (2.25,3.)-- (2.25,0.);\draw (2.5,3.)-- (2.5,0.);\draw (2.75,1.)-- (2.75,0.);
\end{tikzpicture}
\caption{Distribution of the values of $\frac{C_d-A_d}{d}$.}
\label{Fig4}
\end{figure}
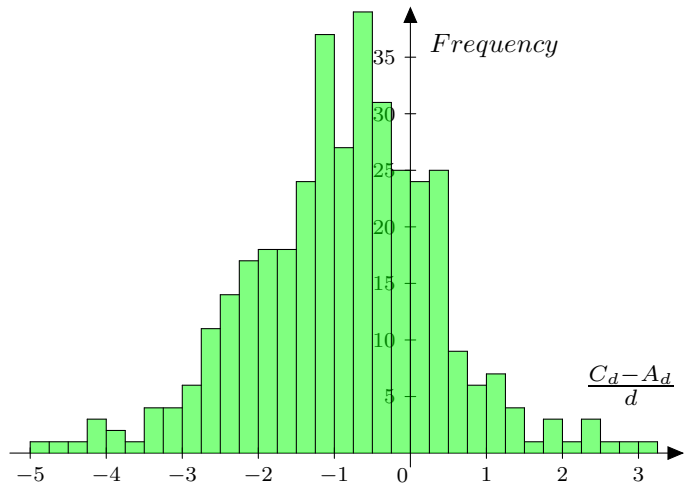

From the result of the first test, we would consider in this part that the function that maps $d$ onto $\frac{C_d-A_d}{d}$ behaves like a 
random variable with an expected value $\Lambda$ close to $-0.856$. On that assumption we will verify whether the values of $\frac{C_d-A_d}{d}$ are 
normally distributed for $d\in[2,2531]\cap\mathbb{N}$ (the null hypothesis) or not. For this purpose we use the Shapiro–Wilk test (see \cite{Sha}). The 
test statistic $W$ is about $0.991$. The associated p-value being about $0.0296$, it can be concluded that we can reject the null hypothesis, 
\textit{i.e.} the values of $\frac{C_d-A_d}{d}$ are significantly not normally distributed for $d\in[2,2531]\cap\mathbb{N}$.

\subsection{A third set of tests}

Once more, from the result of the first test, we would consider in this part that the function that maps $d$ onto $\frac{C_d-A_d}{d}$ behaves like a 
random variable with an expected value $\Lambda$ close to $-0.856$ and with a symmetric probability distribution. 

On that assumption we will verify whether the values higher than $\Lambda$ and those smaller than $\Lambda$ are randomly scattered over the ordered 
absolute values of $\frac{C_d-A_d}{d}$ (the null hypothesis) or not. To this end we use a non-parametric test, the Mann–Whitney $U$ test: we determine 
the ranks of $|\frac{C_d-A_d}{d}|$ for each $d$ in the considered interval (see \cite{W} or \cite{MW}). The ranks sum of the values higher than 
$\Lambda$ is approximately normally distributed. The value of $U_1$ is about $-0.911$. The acceptance region of the hypothesis test with a $5\%$ risk 
being approximately $[-1.960,1.960]$, it can be concluded that we cannot reject the null hypothesis, \textit{i.e.} the fact that the greater and smaller 
than $\Lambda$ values of $\frac{C_d-A_d}{d}$ for $d\in[2,2531]\cap\mathbb{N}$ are randomly scattered: the symmetry of the probability distribution of 
our potential pseudorandom variable can not be rejected.

Once more, on our first assumption, we will verify whether the values higher than $\Lambda$ and those smaller than $\Lambda$ are randomly scattered over 
the first $370$ prime numbers (the null hypothesis) or not. To this end we use the same test, the Mann–Whitney $U$ test. The prime number ranks sum of 
the values higher than $\Lambda$ is approximately normally distributed. The value of $U_2$ is about $-0.397$. It can be concluded that we cannot reject 
the null hypothesis, \textit{i.e.} the fact that the greater and smaller than $\Lambda$ values of $\frac{C_d-A_d}{d}$ for $d\in[2,2531]\cap\mathbb{N}$ 
are randomly scattered over the first $370$ prime numbers.

\subsection{Perspective}

Both first test and set of tests could not invalidate the fact that the function that maps $d$ onto $\frac{C_d-A_d}{d}$ seem to behave like a 
random variable with $\Lambda$ as expected value. If we succeed in proving such a statement, we could consider the following unbiased estimator of $C_d$, 
denoted by $\widehat{C_d}$:
\begin{align*}
\widehat{C_d}&=A_d+\Lambda d\\
&=d^2-d(d+1)(1-\frac{1}{d})^d+\Lambda d\\
&=(1-\frac{1}{e})d^2+(\Lambda-\frac{1}{2e})d + o(d),\quad\text{as $d\rightarrow\infty$}, 
\end{align*}
thanks to the proof of \cite[Theorem 1.]{ste}.

\section{Complexity \textit{versus} number of all but simple points in particular arrangements}

\subsection{Further details on the minimal Besicovitch arrangement providing the complexity}\label{ss1}

In this part, we will consider particular minimal Besicovitch arrangements which share several geometrical properties with 
the arrangements considered to determine the complexity and we will compare the expected values of the number of all but simple points in such 
randomly selected arrangements with $C_d$ values. 

Before reviewing the whole cycles highlighted in section \ref{s1}, let us make a quick remark:

\begin{remark}\label{r1}
If a line in an orbit passes through $(0,0)\in{\mathbb{F}_d}^2$ all the other lines of this orbit also pass through this point, the elements of the 
group $\Gamma$ acting on the lines being in $GL_2(\mathbb{F}_d)$.
\end{remark}

In section \ref{s1} two cases appear, for $d\ge 5$: the cases where $d \equiv 1 \mod 3$ and those where $d \equiv 2 \mod 3$. 

In both cases, the intercepts of the lines in $\{L_0,L_{-1},L_\infty\}$ are $0$ (we have $p(0)=0$ thanks to equality \ref{equ1}, Remark \ref{r1} 
allowing us to conclude). 

The intercepts of the lines in $\{L_1,L_{\frac{d-1}{2}},L_{-2}\}$ are non zero values, except for $d=1093$, the first Wieferich prime number, for which 
lines intercepts are all zero: $p(2)=0\enspace\Longleftrightarrow\enspace\frac{2^{d-1}-1}{d}$ (see equality \ref{equ1}, Remark \ref{r1} and 
\cite{Dorais-Klyve}).

If $d\equiv 1 \mod 3$, the intercepts of the lines in $\{L_\omega,L_{\omega^2}\}$ are $0$: 
\begin{align*}
p(\omega)=&\frac{(\omega+1)^{d}-\omega^d-1}{d}=\frac{-(\omega^{d})^2-\omega^d-1}{d}\\
&=-\frac{-(\omega)^2-\omega-1}{d}=0,
\end{align*}
using the fact that $\omega$ is a primitive cubic root of unity in $\mathbb{F}_d$ and using Fermat's little theorem.

In this part, we will only consider the values of $d\in[2,2531]\cap\mathbb{N}$ where all lines in the $6$-cycles do not pass through $(0,0)$; 
for the $152$ values of $d$ verifying this constraint and also $d\equiv 1 \mod 3$, we denote by $M^*_{d}$ the expected value of the number of all but 
the simple points in a randomly chosen arrangement sharing geometrical properties with the arrangement providing the complexity; for the $153$ values 
of $d$ verifying the same constraint and also $d\equiv 2 \mod 3$, we denote by $M^{**}_{d}$ the similar expected value. Table \ref{Comp} shows the 
values of $d$ being in either the first or the second case.

\subsection{Lines intersections of the different cycles}

The five functions in $\Gamma$, other than the identity function $Id$, will be denote as in \cite{PV}: 
\begin{align*}
&\forall (x,y)\in{\mathbb{F}_d}^2,\\
\iota(x,y)&=(y,x)\qquad\theta(x,y)=(y-x,-x)\qquad\theta^2(x,y)=(-y,x-y)\\
\kappa(x,y)&=\theta\circ\iota(x,y)=(x-y,-y)\qquad\lambda(x,y)=\iota\circ\theta(x,y)=(-x,y-x).
\end{align*}

Note that $\iota$, $\kappa$ and $\lambda$ are of order $2$, and $\theta$ and $\theta^2$ are of order $3$. We can also easily verify that the fixed 
points of $\iota$ are those of the line $L_\infty$, the fixed points of $\kappa$ are those of the line $L_0$ and the fixed points of $\lambda$ are those 
of the line $L_{-1}$. The following proposition can thus be enunciated:

\begin{proposition}\label{prop1}
$\forall \gamma\in\{\iota,\kappa,\lambda\}$ and $\forall a\in\mathbb{F}_d\setminus\{0,-1\}$, if $L_a$ and $\gamma(L_a)$ are two distinct lines, then 
their intersection point is in line of the fixed points of $\gamma$.
\end{proposition}

\begin{proof}
The image of a point under a fonction in $\Gamma\subset GL_2(\mathbb{F}_d)$ is a point. So, $\forall \gamma\in\{\iota,\kappa,\lambda\}$ and $\forall 
a\in\mathbb{F}_d\setminus\{0,-1\}$, if $L_a$ and $\gamma(L_a)$ are two distinct lines, \textit{i.e.} if their intersection is a point: 
\begin{align*}
\gamma(L_a\cap\gamma(L_a))&=\gamma(L_a)\cap\gamma(\gamma(L_a))\\
&=L_a\cap\gamma(L_a),
\end{align*}
each of the considered functions being of order $2$. The point $L_a\cap\gamma(L_a)$ is thus in the fixed line of $\gamma$.
\end{proof}

Let us henceforth denote by $\mathscr{T}$ the set ${\mathbb{F}_d}^2\setminus\{L_0,L_{-1},L_\infty\}$. In each $6$-cycle, for all $\gamma\in\Gamma$ and 
for all $a\in\mathbb{F}_d$ (such that $L_a$ is in the considered $6$-cycle), $L_a$ and $\gamma(L_a)$ are distinct; we can therefore apply Proposition 
\ref{prop1}: in the case where all the lines in a $6$-cycle do not pass through $(0,0)$ (the prevalent selected case in subsection \ref{ss1}), there 
exist $3$ intersection points of the $6$-cycle lines on each line of $\{L_0,L_{-1},L_\infty\}$:
\begin{enumerate}[label=$\bullet$]
 \item on $L_0$: $L_a\cap \kappa(L_a)$, $\theta(L_a)\cap\lambda(L_a)$ and $\theta^2(L_a)\cap\iota(L_a)$;
 \item on $L_{-1}$: $L_a\cap \lambda(L_a)$, $\theta(L_a)\cap\iota(L_a)$ and $\theta^2(L_a)\cap\kappa(L_a)$;
 \item on $L_\infty$: $L_a\cap \iota(L_a)$, $\theta(L_a)\cap\kappa(L_a)$ and $\theta^2(L_a)\cap\lambda(L_a)$.
\end{enumerate}

An other proposition can be added:
\begin{proposition}\label{prop2}
In the case where all the lines in a $6$-cycle do not pass through $(0,0)$, two of the described above $6$-cycle intersection points on a line of 
$\{L_0,L_{-1},L_\infty\}$ do not coincide.
\end{proposition}

\begin{proof}
Let us consider a $6$-cycle. Its lines do not pass through the origin. $a\in\mathbb{F}_d$, such that $L_a$ is in this $6$-cycle. 
We assume that $L_a\cap \kappa(L_a)$ and $\theta(L_a)\cap\lambda(L_a)$ coincide on $L_0$. Knowing that $\theta(L_0)=L_\infty$ (see \cite{PV}), we have: 
\begin{align*}
&\theta(L_a\cap \kappa(L_a)\cap\theta(L_a)\cap\lambda(L_a))\in L_\infty\\
&\theta(L_a)\cap \lambda(L_a)\cap\theta^2(L_a)\cap\iota(L_a)\in L_\infty.
\end{align*}
So $\theta(L_a)\cap\lambda\in L_0\cap L_\infty={(0,0)}$; it contradicts the hypothesis of the proposition. The considered points do not coincide.\\
All the other cases can be demonstrated in the same way.
\end{proof}

Thus the remaining $6$ intersection points of the $6$-cycle lines are in $\mathscr{T}$. We can finally prove the following proposition (in the case 
where $d\ge11$, otherwise there is no $6$-cycle in the arrangement $\mathscr{L}$):

\begin{proposition}\label{prop3}
The $6$ remaining point in $\mathscr{T}$ (in the case where all the lines in the $6$-cycle do not pass through $(0,0)$) are distinct. 
\end{proposition}

\begin{proof}
We first prove the following lemma:
\begin{lemma}\label{l1}
$\theta$ has a single fixed point in ${\mathbb{F}_d}^2\quad\Longleftrightarrow\quad d\neq3$.
\end{lemma}
\begin{proof}
$\theta\in GL_2(\mathbb{F}_d)$ then $(0,0)$ is a fixed point of $\theta$.\\
For $(x,y)\in{\mathbb{F}_d}^2$:
\begin{align*}
\theta(x,y)=(x,y)\quad&\Longleftrightarrow\quad y-x=x\text{ and }-x=y \\
&\Longleftrightarrow\quad 3x=0\text{ and }y=-x.
\end{align*}
The result follows.
\end{proof}
Let us consider a $6$-cycle. Its lines do not pass through the origin. $a\in\mathbb{F}_d$, such that $L_a$ is in this $6$-cycle.\\
Let us assume that $3$ lines in the $6$-cycle are concurrent in $P\in\mathscr{T}$. It is clear from the foregoing that these lines are whether 
$L_a$, $\theta(L_a)$ and $\theta^2(L_a)$ or $\iota(L_a)$, $\kappa(L_a)$ and $\lambda(L_a)$. We have:
\begin{align*}
\theta(P)&=\theta(L_a\cap\theta(L_a)\cap\theta^2(L_a))\quad\big(\text{or}\quad\theta(\iota(L_a)\cap\kappa(L_a)\cap\lambda(L_a))\big)\\
&=\theta(L_a)\cap\theta^2(L_a)\cap L_a\quad\big(\text{or}\quad\kappa(L_a)\cap\lambda(L_a)\cap\iota(L_a)\big).
\end{align*}
In both cases $\theta(P)=P$ \textit{i.e.} $P$ is a fixed point of $\theta$. It means that $P=(0,0)$ thanks to Lemma \ref{l1}, knowing that $d\ge11$; 
it contradicts the hypothesis of the proposition. The $6$ remaining point in $\mathscr{T}$ are distinct.
\end{proof}

The cases of $\{L_1,L_{\frac{d-1}{2}},L_{-2}\}$ and $\{L_\omega,L_{\omega^2}\}$ can be considered as degenerate cases of a $6$-cycle. Let us focus on 
the first arrangement. From \cite{PV}, we get $\iota(L_1)=L_{-2}$, $\lambda(L_{-2})=L_{\frac{d-1}{2}}$ and $\kappa(L_1)=L_{\frac{d-1}{2}}$. Thanks to 
Proposition \ref{prop1}, the $3$ intersection points of lines in $\{L_1,L_{\frac{d-1}{2}},L_{-2}\}$ are:
\begin{enumerate}[label=$\bullet$]
 \item on $L_0$: $L_1\cap L_{\frac{d-1}{2}}$;
 \item on $L_{-1}$: $L_{-2}\cap L_{\frac{d-1}{2}}$;
 \item on $L_\infty$: $L_1\cap L_{-2}$.
\end{enumerate}
We note that this result is just a particular case of the above result.

The Figure \ref{Fig2} below provides two examples of minimal Besicovitch arrangements leading to the determination of the complexity. For the first one 
($d=7$), we are in the case where $d \equiv 1 \mod 3$, for the second one ($d=11$) in the case where $d \equiv 2 \mod 3$. The above results and in 
particular those of Propositions \ref{prop1}, \ref{prop2} and \ref{prop3} are emphasised.

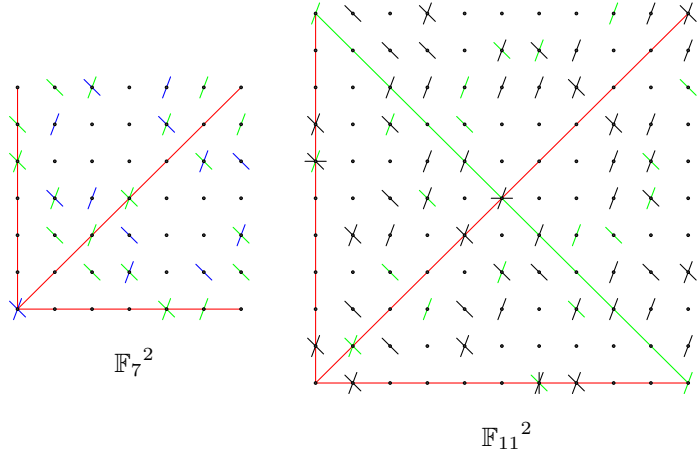
\begin{figure}[ht]
\centering
\definecolor{qqffqq}{rgb}{0.,1.,0.}
\definecolor{qqqqff}{rgb}{0.,0.,1.}
\definecolor{ffqqqq}{rgb}{1.,0.,0.}
\definecolor{uuuuuu}{rgb}{0.26666666666666666,0.26666666666666666,0.26666666666666666}
\begin{tikzpicture}[line cap=round,line join=round,>=triangle 45,x=0.7cm,y=0.7cm]
\clip(-0.872575340805852,-2) rectangle (15.577138254236258,7);
\draw [color=ffqqqq] (0.7,4.9)-- (0.7,0.7);\draw [color=ffqqqq] (0.7,0.7)-- (4.9,0.7);\draw [color=ffqqqq] (0.7,0.7)-- (4.9,4.9);
\draw [color=ffqqqq] (6.3,6.3)-- (6.3,-0.7);\draw [color=ffqqqq] (6.3,-0.7)-- (13.3,-0.7);\draw [color=ffqqqq] (6.3,-0.7)-- (13.3,6.3);
\draw [color=qqqqff] (0.55,0.85)-- (0.85,0.55);\draw [color=qqqqff] (1.25,2.95)-- (1.55,2.65);\draw [color=qqqqff] (1.95,5.05)-- (2.25,4.75);
\draw [color=qqqqff] (2.65,2.25)-- (2.95,1.95);\draw [color=qqqqff] (3.35,4.35)-- (3.65,4.05);\draw [color=qqqqff] (4.05,1.55)-- (4.35,1.25);
\draw [color=qqqqff] (4.75,3.65)-- (5.05,3.35);\draw [color=qqqqff] (0.625,0.5)-- (0.775,0.9);\draw [color=qqqqff] (1.325,4.)-- (1.475,4.4);
\draw [color=qqqqff] (2.025,2.6)-- (2.175,3.);\draw [color=qqqqff] (2.725,1.2)-- (2.875,1.6);\draw [color=qqqqff] (3.425,4.7)-- (3.575,5.1);
\draw [color=qqqqff] (4.125,3.3)-- (4.275,3.7);\draw [color=qqqqff] (4.825,1.9)-- (4.975,2.3);\draw [color=qqffqq] (0.55,4.35)-- (0.85,4.05);
\draw [color=qqffqq] (1.25,2.25)-- (1.55,1.95);\draw [color=qqffqq] (2.025,4.7)-- (2.175,5.1);\draw [color=qqffqq] (2.725,2.6)-- (2.875,3.);
\draw [color=qqffqq] (3.425,0.5)-- (3.575,0.9);\draw [color=qqffqq] (4.05,3.65)-- (4.35,3.35);\draw [color=qqffqq] (4.75,1.55)-- (5.05,1.25);
\draw [color=qqffqq] (0.625,3.3)-- (0.775,3.7);\draw [color=qqffqq] (1.325,2.6)-- (1.475,3.);\draw [color=qqffqq] (2.025,1.9)-- (2.175,2.3);
\draw [color=qqffqq] (2.65,1.55)-- (2.95,1.25);\draw [color=qqffqq] (3.35,0.85)-- (3.65,0.55);\draw [color=qqffqq] (4.125,4.7)-- (4.275,5.1);
\draw [color=qqffqq] (4.825,4.)-- (4.975,4.4);\draw [color=qqffqq] (0.55,3.65)-- (0.85,3.35);\draw [color=qqffqq] (1.25,5.05)-- (1.55,4.75);
\draw [color=qqffqq] (1.95,1.55)-- (2.25,1.25);\draw [color=qqffqq] (2.65,2.95)-- (2.95,2.65);\draw [color=qqffqq] (3.425,4.)-- (3.575,4.4);
\draw [color=qqffqq] (4.125,0.5)-- (4.275,0.9);\draw [color=qqffqq] (4.75,2.25)-- (5.05,1.95);\draw [color=qqffqq] (6.225,3.3)-- (6.375,3.7);
\draw [color=qqffqq] (6.925,-0.2)-- (7.075,0.2);\draw [color=qqffqq] (7.625,4.)-- (7.775,4.4);\draw [color=qqffqq] (9.025,4.7)-- (9.175,5.1);
\draw [color=qqffqq] (10.425,5.4)-- (10.575,5.8);\draw [color=qqffqq] (11.825,6.1)-- (11.975,6.5);\draw [color=qqffqq] (8.325,0.5)-- (8.475,0.9);
\draw [color=qqffqq] (9.725,1.2)-- (9.875,1.6);\draw [color=qqffqq] (11.125,1.9)-- (11.275,2.3);\draw [color=qqffqq] (12.525,2.6)-- (12.675,3.);
\draw [color=qqffqq] (13.225,-0.9)-- (13.375,-0.5);\draw [color=qqffqq] (6.3,6.3)-- (13.3,-0.7);\draw [color=qqffqq] (6.225,6.1)-- (6.375,6.5);
\draw [color=qqffqq] (6.85,0.15)-- (7.15,-0.15);\draw [color=qqffqq] (7.55,1.55)-- (7.85,1.25);\draw [color=qqffqq] (8.25,2.95)-- (8.55,2.65);
\draw [color=qqffqq] (8.95,4.35)-- (9.25,4.05);\draw [color=qqffqq] (9.65,5.75)-- (9.95,5.45);\draw [color=qqffqq] (10.35,-0.55)-- (10.65,-0.85);
\draw [color=qqffqq] (11.05,0.85)-- (11.35,0.55);\draw [color=qqffqq] (11.75,2.25)-- (12.05,1.95);\draw [color=qqffqq] (12.45,3.65)-- (12.75,3.35);
\draw [color=qqffqq] (13.15,5.05)-- (13.45,4.75);\draw (6.15,4.35)-- (6.45,4.05);\draw (6.85,2.25)-- (7.15,1.95);\draw (7.55,0.15)-- (7.85,-0.15);
\draw (8.25,5.75)-- (8.55,5.45);\draw (9.025,3.3)-- (9.175,3.7);\draw (9.65,1.55)-- (9.95,1.25);\draw (10.425,-0.9)-- (10.575,-0.5);
\draw (11.125,4.7)-- (11.275,5.1);\draw (11.825,2.6)-- (11.975,3.);\draw (12.525,0.5)-- (12.675,0.9);\draw (13.225,6.1)-- (13.375,6.5);
\draw (6.15,0.15)-- (6.45,-0.15);\draw (6.85,6.45)-- (7.15,6.15);\draw (7.625,4.7)-- (7.775,5.1);\draw (8.325,3.3)-- (8.475,3.7);
\draw (9.025,1.9)-- (9.175,2.3);\draw (9.725,0.5)-- (9.875,0.9);\draw (10.5,-0.5)-- (10.5,-0.9);\draw (11.125,5.4)-- (11.275,5.8);
\draw (11.825,4.)-- (11.975,4.4);\draw (12.45,2.95)-- (12.75,2.65);\draw (13.15,1.55)-- (13.45,1.25);\draw (6.225,-0.2)-- (6.375,0.2);
\draw (6.925,1.9)-- (7.075,2.3);\draw (7.55,4.35)-- (7.85,4.05);\draw (8.25,6.45)-- (8.55,6.15);\draw (8.95,0.85)-- (9.25,0.55);
\draw (9.725,2.6)-- (9.875,3.);\draw (10.425,4.7)-- (10.575,5.1);\draw (11.125,-0.9)-- (11.275,-0.5);\draw (11.825,1.2)-- (11.975,1.6);
\draw (12.525,3.3)-- (12.675,3.7);\draw (13.225,5.4)-- (13.375,5.8);\draw (6.225,4.)-- (6.375,4.4);\draw (6.925,-0.9)-- (7.075,-0.5);
\draw (7.625,1.9)-- (7.775,2.3);\draw (8.325,4.7)-- (8.475,5.1);\draw (9.025,-0.2)-- (9.175,0.2);\draw (9.6,2.8)-- (10.,2.8);
\draw (10.35,5.75)-- (10.65,5.45);\draw (11.125,0.5)-- (11.275,0.9);\draw (11.825,3.3)-- (11.975,3.7);\draw (12.525,6.1)-- (12.675,6.5);
\draw (13.225,1.2)-- (13.375,1.6);\draw (6.15,3.65)-- (6.45,3.35);\draw (6.85,-0.55)-- (7.15,-0.85);\draw (7.55,2.95)-- (7.85,2.65);
\draw (8.325,6.1)-- (8.475,6.5);\draw (8.95,2.25)-- (9.25,1.95);\draw (10.35,1.55)-- (10.65,1.25);\draw (11.825,0.5)-- (11.975,0.9);
\draw (13.225,-0.2)-- (13.375,0.2);\draw (9.725,5.4)-- (9.875,5.8);\draw (11.05,5.05)-- (11.35,4.75);\draw (12.525,4.)-- (12.675,4.4);
\draw (6.1,3.5)-- (6.5,3.5);\draw (6.85,0.85)-- (7.15,0.55);\draw (7.55,5.75)-- (7.85,5.45);\draw (8.325,2.6)-- (8.475,3.);
\draw (8.95,0.15)-- (9.25,-0.15);\draw (9.725,4.7)-- (9.875,5.1);\draw (10.425,1.9)-- (10.575,2.3);\draw (11.05,-0.55)-- (11.35,-0.85);
\draw (11.75,4.35)-- (12.05,4.05);\draw (12.45,1.55)-- (12.75,1.25);\draw (13.15,6.45)-- (13.45,6.15);
\draw (2.9,-0.3) node {${\mathbb{F}_7}^2$};\draw (9.9,-1.7) node {${\mathbb{F}_{11}}^2$};
\begin{scriptsize}
\draw [fill=uuuuuu] (0.7,0.7) circle (0.5pt);\draw [fill=uuuuuu] (1.4,0.7) circle (0.5pt);
\draw [fill=uuuuuu] (2.1,0.7) circle (0.5pt);\draw [fill=uuuuuu] (2.8,0.7) circle (0.5pt);\draw [fill=uuuuuu] (3.5,0.7) circle (0.5pt);
\draw [fill=uuuuuu] (4.2,0.7) circle (0.5pt);\draw [fill=uuuuuu] (4.9,0.7) circle (0.5pt);\draw [fill=uuuuuu] (0.7,1.4) circle (0.5pt);
\draw [fill=uuuuuu] (1.4,1.4) circle (0.5pt);\draw [fill=uuuuuu] (2.1,1.4) circle (0.5pt);\draw [fill=uuuuuu] (2.8,1.4) circle (0.5pt);
\draw [fill=uuuuuu] (3.5,1.4) circle (0.5pt);\draw [fill=uuuuuu] (4.2,1.4) circle (0.5pt);\draw [fill=uuuuuu] (4.9,1.4) circle (0.5pt);
\draw [fill=uuuuuu] (0.7,2.1) circle (0.5pt);\draw [fill=uuuuuu] (1.4,2.1) circle (0.5pt);\draw [fill=uuuuuu] (2.1,2.1) circle (0.5pt);
\draw [fill=uuuuuu] (2.8,2.1) circle (0.5pt);\draw [fill=uuuuuu] (3.5,2.1) circle (0.5pt);\draw [fill=uuuuuu] (4.2,2.1) circle (0.5pt);
\draw [fill=uuuuuu] (4.9,2.1) circle (0.5pt);\draw [fill=uuuuuu] (0.7,2.8) circle (0.5pt);\draw [fill=uuuuuu] (1.4,2.8) circle (0.5pt);
\draw [fill=uuuuuu] (2.1,2.8) circle (0.5pt);\draw [fill=uuuuuu] (2.8,2.8) circle (0.5pt);\draw [fill=uuuuuu] (3.5,2.8) circle (0.5pt);
\draw [fill=uuuuuu] (4.2,2.8) circle (0.5pt);\draw [fill=uuuuuu] (4.9,2.8) circle (0.5pt);\draw [fill=uuuuuu] (0.7,3.5) circle (0.5pt);
\draw [fill=uuuuuu] (1.4,3.5) circle (0.5pt);\draw [fill=uuuuuu] (2.1,3.5) circle (0.5pt);\draw [fill=uuuuuu] (2.8,3.5) circle (0.5pt);
\draw [fill=uuuuuu] (3.5,3.5) circle (0.5pt);\draw [fill=uuuuuu] (4.2,3.5) circle (0.5pt);\draw [fill=uuuuuu] (4.9,3.5) circle (0.5pt);
\draw [fill=uuuuuu] (0.7,4.2) circle (0.5pt);\draw [fill=uuuuuu] (1.4,4.2) circle (0.5pt);\draw [fill=uuuuuu] (2.1,4.2) circle (0.5pt);
\draw [fill=uuuuuu] (2.8,4.2) circle (0.5pt);\draw [fill=uuuuuu] (3.5,4.2) circle (0.5pt);\draw [fill=uuuuuu] (4.2,4.2) circle (0.5pt);
\draw [fill=uuuuuu] (4.9,4.2) circle (0.5pt);\draw [fill=uuuuuu] (0.7,4.9) circle (0.5pt);\draw [fill=uuuuuu] (1.4,4.9) circle (0.5pt);
\draw [fill=uuuuuu] (2.1,4.9) circle (0.5pt);\draw [fill=uuuuuu] (2.8,4.9) circle (0.5pt);\draw [fill=uuuuuu] (3.5,4.9) circle (0.5pt);
\draw [fill=uuuuuu] (4.2,4.9) circle (0.5pt);\draw [fill=uuuuuu] (4.9,4.9) circle (0.5pt);
\draw [fill=uuuuuu] (6.3,-0.7) circle (0.5pt);\draw [fill=uuuuuu] (7.,-0.7) circle (0.5pt);\draw [fill=uuuuuu] (7.7,-0.7) circle (0.5pt);
\draw [fill=uuuuuu] (8.4,-0.7) circle (0.5pt);\draw [fill=uuuuuu] (9.1,-0.7) circle (0.5pt);\draw [fill=uuuuuu] (9.8,-0.7) circle (0.5pt);
\draw [fill=uuuuuu] (10.5,-0.7) circle (0.5pt);\draw [fill=uuuuuu] (11.2,-0.7) circle (0.5pt);\draw [fill=uuuuuu] (11.9,-0.7) circle (0.5pt);
\draw [fill=uuuuuu] (12.6,-0.7) circle (0.5pt);\draw [fill=uuuuuu] (13.3,-0.7) circle (0.5pt);\draw [fill=uuuuuu] (6.3,0.) circle (0.5pt);
\draw [fill=uuuuuu] (7.,0.) circle (0.5pt);\draw [fill=uuuuuu] (7.7,0.) circle (0.5pt);\draw [fill=uuuuuu] (8.4,0.) circle (0.5pt);
\draw [fill=uuuuuu] (9.1,0.) circle (0.5pt);\draw [fill=uuuuuu] (9.8,0.) circle (0.5pt);\draw [fill=uuuuuu] (10.5,0.) circle (0.5pt);
\draw [fill=uuuuuu] (11.2,0.) circle (0.5pt);\draw [fill=uuuuuu] (11.9,0.) circle (0.5pt);\draw [fill=uuuuuu] (12.6,0.) circle (0.5pt);
\draw [fill=uuuuuu] (13.3,0.) circle (0.5pt);\draw [fill=uuuuuu] (6.3,0.7) circle (0.5pt);\draw [fill=uuuuuu] (7.,0.7) circle (0.5pt);
\draw [fill=uuuuuu] (7.7,0.7) circle (0.5pt);\draw [fill=uuuuuu] (8.4,0.7) circle (0.5pt);\draw [fill=uuuuuu] (9.1,0.7) circle (0.5pt);
\draw [fill=uuuuuu] (9.8,0.7) circle (0.5pt);\draw [fill=uuuuuu] (10.5,0.7) circle (0.5pt);\draw [fill=uuuuuu] (11.2,0.7) circle (0.5pt);
\draw [fill=uuuuuu] (11.9,0.7) circle (0.5pt);\draw [fill=uuuuuu] (12.6,0.7) circle (0.5pt);\draw [fill=uuuuuu] (13.3,0.7) circle (0.5pt);
\draw [fill=uuuuuu] (6.3,1.4) circle (0.5pt);\draw [fill=uuuuuu] (7.,1.4) circle (0.5pt);\draw [fill=uuuuuu] (7.7,1.4) circle (0.5pt);
\draw [fill=uuuuuu] (8.4,1.4) circle (0.5pt);\draw [fill=uuuuuu] (9.1,1.4) circle (0.5pt);\draw [fill=uuuuuu] (9.8,1.4) circle (0.5pt);
\draw [fill=uuuuuu] (10.5,1.4) circle (0.5pt);\draw [fill=uuuuuu] (11.2,1.4) circle (0.5pt);\draw [fill=uuuuuu] (11.9,1.4) circle (0.5pt);
\draw [fill=uuuuuu] (12.6,1.4) circle (0.5pt);\draw [fill=uuuuuu] (13.3,1.4) circle (0.5pt);\draw [fill=uuuuuu] (6.3,2.1) circle (0.5pt);
\draw [fill=uuuuuu] (7.,2.1) circle (0.5pt);\draw [fill=uuuuuu] (7.7,2.1) circle (0.5pt);\draw [fill=uuuuuu] (8.4,2.1) circle (0.5pt);
\draw [fill=uuuuuu] (9.1,2.1) circle (0.5pt);\draw [fill=uuuuuu] (9.8,2.1) circle (0.5pt);\draw [fill=uuuuuu] (10.5,2.1) circle (0.5pt);
\draw [fill=uuuuuu] (11.2,2.1) circle (0.5pt);\draw [fill=uuuuuu] (11.9,2.1) circle (0.5pt);\draw [fill=uuuuuu] (12.6,2.1) circle (0.5pt);
\draw [fill=uuuuuu] (13.3,2.1) circle (0.5pt);\draw [fill=uuuuuu] (6.3,2.8) circle (0.5pt);\draw [fill=uuuuuu] (7.,2.8) circle (0.5pt);
\draw [fill=uuuuuu] (7.7,2.8) circle (0.5pt);\draw [fill=uuuuuu] (8.4,2.8) circle (0.5pt);\draw [fill=uuuuuu] (9.1,2.8) circle (0.5pt);
\draw [fill=uuuuuu] (9.8,2.8) circle (0.5pt);\draw [fill=uuuuuu] (10.5,2.8) circle (0.5pt);\draw [fill=uuuuuu] (11.2,2.8) circle (0.5pt);
\draw [fill=uuuuuu] (11.9,2.8) circle (0.5pt);\draw [fill=uuuuuu] (12.6,2.8) circle (0.5pt);\draw [fill=uuuuuu] (13.3,2.8) circle (0.5pt);
\draw [fill=uuuuuu] (6.3,3.5) circle (0.5pt);\draw [fill=uuuuuu] (7.,3.5) circle (0.5pt);\draw [fill=uuuuuu] (7.7,3.5) circle (0.5pt);
\draw [fill=uuuuuu] (8.4,3.5) circle (0.5pt);\draw [fill=uuuuuu] (9.1,3.5) circle (0.5pt);\draw [fill=uuuuuu] (9.8,3.5) circle (0.5pt);
\draw [fill=uuuuuu] (10.5,3.5) circle (0.5pt);\draw [fill=uuuuuu] (11.2,3.5) circle (0.5pt);\draw [fill=uuuuuu] (11.9,3.5) circle (0.5pt);
\draw [fill=uuuuuu] (12.6,3.5) circle (0.5pt);\draw [fill=uuuuuu] (13.3,3.5) circle (0.5pt);\draw [fill=uuuuuu] (6.3,4.2) circle (0.5pt);
\draw [fill=uuuuuu] (7.,4.2) circle (0.5pt);\draw [fill=uuuuuu] (7.7,4.2) circle (0.5pt);\draw [fill=uuuuuu] (8.4,4.2) circle (0.5pt);
\draw [fill=uuuuuu] (9.1,4.2) circle (0.5pt);\draw [fill=uuuuuu] (9.8,4.2) circle (0.5pt);\draw [fill=uuuuuu] (10.5,4.2) circle (0.5pt);
\draw [fill=uuuuuu] (11.2,4.2) circle (0.5pt);\draw [fill=uuuuuu] (11.9,4.2) circle (0.5pt);\draw [fill=uuuuuu] (12.6,4.2) circle (0.5pt);
\draw [fill=uuuuuu] (13.3,4.2) circle (0.5pt);\draw [fill=uuuuuu] (6.3,4.9) circle (0.5pt);\draw [fill=uuuuuu] (7.,4.9) circle (0.5pt);
\draw [fill=uuuuuu] (7.7,4.9) circle (0.5pt);\draw [fill=uuuuuu] (8.4,4.9) circle (0.5pt);\draw [fill=uuuuuu] (9.1,4.9) circle (0.5pt);
\draw [fill=uuuuuu] (9.8,4.9) circle (0.5pt);\draw [fill=uuuuuu] (10.5,4.9) circle (0.5pt);\draw [fill=uuuuuu] (11.2,4.9) circle (0.5pt);
\draw [fill=uuuuuu] (11.9,4.9) circle (0.5pt);\draw [fill=uuuuuu] (12.6,4.9) circle (0.5pt);\draw [fill=uuuuuu] (13.3,4.9) circle (0.5pt);
\draw [fill=uuuuuu] (6.3,5.6) circle (0.5pt);\draw [fill=uuuuuu] (7.,5.6) circle (0.5pt);\draw [fill=uuuuuu] (7.7,5.6) circle (0.5pt);
\draw [fill=uuuuuu] (8.4,5.6) circle (0.5pt);\draw [fill=uuuuuu] (9.1,5.6) circle (0.5pt);\draw [fill=uuuuuu] (9.8,5.6) circle (0.5pt);
\draw [fill=uuuuuu] (10.5,5.6) circle (0.5pt);\draw [fill=uuuuuu] (11.2,5.6) circle (0.5pt);\draw [fill=uuuuuu] (11.9,5.6) circle (0.5pt);
\draw [fill=uuuuuu] (12.6,5.6) circle (0.5pt);\draw [fill=uuuuuu] (13.3,5.6) circle (0.5pt);\draw [fill=uuuuuu] (6.3,6.3) circle (0.5pt);
\draw [fill=uuuuuu] (7.,6.3) circle (0.5pt);\draw [fill=uuuuuu] (7.7,6.3) circle (0.5pt);\draw [fill=uuuuuu] (8.4,6.3) circle (0.5pt);
\draw [fill=uuuuuu] (9.1,6.3) circle (0.5pt);\draw [fill=uuuuuu] (9.8,6.3) circle (0.5pt);\draw [fill=uuuuuu] (10.5,6.3) circle (0.5pt);
\draw [fill=uuuuuu] (11.2,6.3) circle (0.5pt);\draw [fill=uuuuuu] (11.9,6.3) circle (0.5pt);\draw [fill=uuuuuu] (12.6,6.3) circle (0.5pt);
\draw [fill=uuuuuu] (13.3,6.3) circle (0.5pt);
\end{scriptsize}
\end{tikzpicture}
\caption{Lines of the minimal Besicovitch arrangement in ${\mathbb{F}_d}^2$ providing the complexity $C_d$ where $d=7$ (on the left) and 
$d=11$ (on the right). Red lines are those of $\{L_0,L_{-1},L_\infty\}$, green ones are those of $\{L_1,L_{\frac{d-1}{2}},L_{-2}\}$, blue ones are those 
of $\{L_\omega,L_{\omega^2}\}$ and black ones are lines of a $6$-cycle. The number of all but the simple points is $25$ for $d=7$, and $67$ for 
$d=11$; thus $C_7=25$ and $C_{11}=67$.}
\label{Fig2}
\end{figure}

\subsection{The first model}

We first consider the case where $d\equiv1\mod3$. Let us denote by $\Omega^*$ the set of minimal Besicovitch arrangements verifying some geometrical 
constraints similar to those of the considered Besicovitch arrangements. In such arrangements:
\begin{enumerate}[label=$\bullet$]
 \item there exist $3$ lines of equations $x=0$, $y=0$ and $y=x$ (let us denote by $l_a$ this lines set);
 \item there exist $2$ lines that pass through the origin (let us denote by $l_2$ this lines set);
 \item there exist $3$ lines that do not pass through the origin, their $3$ intersection points being respectively on each of the $3$ lines in $l_a$ 
 (let us denote by $l_3$ this lines set);
 \item there exist $\frac{d-7}{6}$ sets of $6$ lines, all verifying the same constraints as in Propositions \ref{prop1}, \ref{prop2} and \ref{prop3}.
\end{enumerate}

In order to calculate the average number of all but simple points in such arrangements, we will build a probability space: $\Omega^*$. The 
$\sigma$-algebra chosen here is the finite collection of all subsets of $\Omega^*$. Our probability measure, denoted by $\p$, assigns equal 
probabilities to all outcomes. 

For $Q$ in ${\mathbb{F}_d}^2$, let $M_{Q}$ be the random variable that maps $A\in\Omega^*$ to the multiplicity of $Q$ in $A$.

With the aim of knowing the expected number of simple points in such particular arrangements, we determine $\p(M_Q=1)$, for all $Q$ in 
${\mathbb{F}_d}^2$. Two cases appear: either $Q$ is in a line of $l_a$ (apart from the origin) or not. 

\subsubsection{\texorpdfstring{$Q$}{} is in a line of \texorpdfstring{$l_a$}{} (apart from the origin)}

In this case, for $A\in\Omega^*$, we have: 
\begin{center}
$M_Q(A)=1\quad\Longleftrightarrow\quad$ none of the $d-2$ lines of $A$ (other than those of $l_a$) pass through $Q$. 
\end{center}

We already know that lines of $l_2$ do not pass through this point.

There is a $\frac{d-2}{d-1}\times\frac{d-3}{d-2}$ probability that the two distinct intersection points between lines of $l_3$ and the considered line 
of $l_a$ do not coincide with $Q$ (a line is composed of $d$ points and the origin is here not considered).

Similarly, there is a $\frac{d-2}{d-1}\times\frac{d-3}{d-2}\times\frac{d-4}{d-3}$ probability that the three distinct intersection points between lines 
of a set of $6$ lines (verifying the same constraints as in Propositions \ref{prop1} and \ref{prop2}) and the considered line of $l_a$ do not coincide 
with $Q$. 

Finally, considering the $\frac{d-7}{6}$ sets of $6$ lines and the lines in $l_2$ and $l_3$, we obtain in this case:
\begin{align*}
\p(M_Q=1)=\frac{d-3}{d-1}\times\Big(\frac{d-4}{d-1}\Big)^{\frac{d-7}{6}}.
\end{align*}

\subsubsection{\texorpdfstring{$Q$}{} is not in a line of \texorpdfstring{$l_a$}{}}

In this case, for $A\in\Omega^*$, we have: 
\begin{center}
$M_Q(A)=1\quad\Longleftrightarrow\quad$ exactly one line of the $d-2$ lines of $A$ (other than those of $l_a$) passes through $Q$.
\end{center}

We will use the following results to study in more detail the different subcases. In ${\mathbb{F}_d}^2\setminus l_a$, there are 
$d^2-3\times(d-1)-1=d^2-3d+2$ points. In ${\mathbb{F}_d}^2\setminus l_a\cup l_2$, there are $2(d-1)$ points of multiplicity $1$ and the remaining points 
of multiplicity $0$ ($d^2-5d+4$ points). In ${\mathbb{F}_d}^2\setminus l_a\cup l_3$, there are $3(d-3)$ points of multiplicity $1$ and the remaining 
points of multiplicity $0$ ($d^2-6d+11$ points). In the union of ${\mathbb{F}_d}^2\setminus l_a$ and a $6$ lines set, there are $6(d-5)$ 
points of multiplicity $1$, $6$ points of multiplicity $2$ (see Proposition \ref{prop3}) and the remaining points of multiplicity $0$ ($d^2-9d+26$ 
points).

This case can be divided in $3$ subcases:
\begin{enumerate}[label=$\bullet$]
 \item the first one where the line that passes through $Q$ is in $l_2$; then the probability is:
 \begin{align*}
 \frac{2(d-1)}{d^2-3d+2}\times\frac{d^2-6d+11}{d^2-3d+2}\times\Big(\frac{d^2-9d+26}{d^2-3d+2}\Big)^{\frac{d-7}{6}};
 \end{align*}
 \item the second one where the line that passes through $Q$ is in $l_3$; then the probability is:
 \begin{align*}
 \frac{d^2-5d+4}{d^2-3d+2}\times\frac{3(d-3)}{d^2-3d+2}\times\Big(\frac{d^2-9d+26}{d^2-3d+2}\Big)^{\frac{d-7}{6}};
 \end{align*}
 \item the third one where the line that passes through $Q$ is in one of the $\frac{d-7}{6}$ sets of $6$ lines; then the probability is:
 \begin{align*}
 \frac{d^2-5d+4}{d^2-3d+2}\times\frac{d^2-6d+11}{d^2-3d+2}\times\frac{d-7}{6}\frac{6(d-5)}{d^2-3d+2}\Big(\frac{d^2-9d+26}{d^2-3d+2}\Big)^{\frac{d-13}{6}}.
 \end{align*} 
\end{enumerate}

Hence we have in this specific case:
\begin{align*}
\p(M_Q=1)&=\frac{2}{d-2}\times\frac{d^2-6d+11}{d^2-3d+2}\times\Big(\frac{d^2-9d+26}{d^2-3d+2}\Big)^{\frac{d-7}{6}}\\
&+\frac{d-4}{d-2}\times\frac{3(d-3)}{d^2-3d+2}\times\Big(\frac{d^2-9d+26}{d^2-3d+2}\Big)^{\frac{d-7}{6}}\\
&+\frac{d-4}{d-2}\times\frac{d^2-6d+11}{d^2-3d+2}\times\frac{d^2-12d+35}{d^2-3d+2}\Big(\frac{d^2-9d+26}{d^2-3d+2}\Big)^{\frac{d-13}{6}}.
\end{align*}

\subsubsection{The expected value of \texorpdfstring{$M^*_{d}$}{}}

Recall that our aim is to determine the expected value $M^*_{d}$ of the number of all but simple points in arrangements of $\Omega^*$ in order to 
compare it with the value of the complexity $C_d$.

Thanks to the results of the above section and knowing that the first case concerns $3d-3$ points and the second one $d^2-3d+2$ points, we get:
\begin{align*}
M^*_{d}=&d^2-\Big(3(d-3)\times\big(\frac{d-4}{d-1}\big)^{\frac{d-7}{6}}+\frac{2(d^2-6d+11)}{d-2}\times\big(\frac{d^2-9d+26}{d^2-3d+2}\big)^{\frac{d-7}{6}}\\
&+\frac{3(d-3)(d-4)}{d-2}\times\big(\frac{d^2-9d+26}{d^2-3d+2}\big)^{\frac{d-7}{6}}\\
&+\frac{(d-4)(d^2-6d+11)}{d-2}\times\frac{d^2-12d+35}{d^2-3d+2}\big(\frac{d^2-9d+26}{d^2-3d+2}\big)^{\frac{d-13}{6}}\Big).
\end{align*}

Using the Computer Algebra System Giac/Xcas (Parisse and De Graeve, 2017, \url{http://www-fourier.ujf-grenoble.fr/~parisse/giac_fr.html}, 
version 1.2.3), we obtain:
\begin{align*}
M^*_{d}=(1-\frac{1}{e})d^2+\big(\frac{1}{e}-3\exp(-\frac{1}{2})\big)d+O(1),\quad\text{as $d\rightarrow\infty$}.
\end{align*}

\subsection{The second model}

We henceforth consider the case where $d\equiv2\mod3$. Let us denote by $\Omega^{**}$ the set of minimal Besicovitch arrangements verifying some 
geometrical constraints similar to those of the considered Besicovitch arrangements. In such arrangements:
\begin{enumerate}[label=$\bullet$]
 \item there exist $3$ lines of equations $x=0$, $y=0$ and $y=x$ (let us denote by $l_a$ this lines set);
 \item there exist $3$ lines that do not pass through the origin, their $3$ intersection points being respectively on each of the $3$ lines in $l_a$ 
 (let us denote by $l_3$ this lines set);
 \item there exist $\frac{d-5}{6}$ sets of $6$ lines, all verifying the same constraints as in Propositions \ref{prop1}, \ref{prop2} and \ref{prop3}.
\end{enumerate}

In order to calculate the average number of all but simple points in such arrangements, we will build a probability space: $\Omega^{**}$. The 
$\sigma$-algebra chosen here is the finite collection of all subsets of $\Omega^{**}$. Our probability measure, denoted by $P$, assigns equal 
probabilities to all outcomes. 

For $Q$ in ${\mathbb{F}_d}^2$, let $M_{Q}$ be the random variable that maps $A\in\Omega^{**}$ to the multiplicity of $Q$ in $A$.

With the aim of knowing the expected number of simple points in such particular arrangements, we determine $P(M_Q=1)$, for all $Q$ in 
${\mathbb{F}_d}^2$. Two cases appear: either $Q$ is in a line of $l_a$ (apart from the origin) or not. 

\subsubsection{\texorpdfstring{$Q$}{} is in a line of \texorpdfstring{$l_a$}{} (apart from the origin)}

In this case, for $A\in\Omega^{**}$, we have: 
\begin{center}
$M_Q(A)=1\quad\Longleftrightarrow\quad$ none of the $d-2$ lines of $A$ (other than those of $l_a$) pass through $Q$. 
\end{center}

There is a $\frac{d-2}{d-1}\times\frac{d-3}{d-2}$ probability that the two distinct intersection points between lines of $l_3$ and the considered line 
of $l_a$ do not coincide with $Q$.

Similarly, there is a $\frac{d-2}{d-1}\times\frac{d-3}{d-2}\times\frac{d-4}{d-3}$ probability that the three distinct intersection points between lines 
of a set of $6$ lines and the considered line of $l_a$ do not coincide with $Q$. 

Finally, considering the $\frac{d-5}{6}$ sets of $6$ lines and the lines in $l_3$, we obtain in this case:
\begin{align*}
P(M_Q=1)=\frac{d-3}{d-1}\times\Big(\frac{d-4}{d-1}\Big)^{\frac{d-5}{6}}.
\end{align*}

\subsubsection{\texorpdfstring{$Q$}{} is not in a line of \texorpdfstring{$l_a$}{}}

In this case, for $A\in\Omega^{**}$, we have: 
\begin{center}
$M_Q(A)=1\quad\Longleftrightarrow\quad$ exactly one line of the $d-2$ lines of $A$ (other than those of $l_a$) passes through $Q$.
\end{center}

We will use the following results to study in more detail the different subcases. In ${\mathbb{F}_d}^2\setminus l_a$, there are $d^2-3d+2$ points. In 
${\mathbb{F}_d}^2\setminus l_a\cup l_3$, there are $3(d-3)$ points of multiplicity $1$ and the remaining points of multiplicity $0$ ($d^2-6d+11$ 
points). In the union of ${\mathbb{F}_d}^2\setminus l_a$ and a $6$ lines set, there are $6(d-5)$ points of multiplicity $1$, $6$ points of 
multiplicity $2$ (see Proposition \ref{prop3}) and the remaining points of multiplicity $0$ ($d^2-9d+26$ points).

This case can be divided in $2$ subcases:
\begin{enumerate}[label=$\bullet$]
 \item the first one where the line that passes through $Q$ is in $l_3$; then the probability is:
 \begin{align*}
 \frac{3(d-3)}{d^2-3d+2}\times\Big(\frac{d^2-9d+26}{d^2-3d+2}\Big)^{\frac{d-5}{6}};
 \end{align*}
 \item the second one where the line that passes through $Q$ is in one of the $\frac{d-5}{6}$ sets of $6$ lines; then the probability is:
 \begin{align*}
 \frac{d^2-6d+11}{d^2-3d+2}\times\frac{d-5}{6}\frac{6(d-5)}{d^2-3d+2}\Big(\frac{d^2-9d+26}{d^2-3d+2}\Big)^{\frac{d-11}{6}}.
 \end{align*} 
\end{enumerate}

Hence we have in this specific case:
\begin{align*}
P(M_Q=1)&=\frac{3(d-3)}{d^2-3d+2}\times\Big(\frac{d^2-9d+26}{d^2-3d+2}\Big)^{\frac{d-5}{6}}\\
&+\frac{d^2-6d+11}{d^2-3d+2}\times\frac{d^2-10d+25}{d^2-3d+2}\Big(\frac{d^2-9d+26}{d^2-3d+2}\Big)^{\frac{d-11}{6}}.
\end{align*}

\subsubsection{The expected value of \texorpdfstring{$M^{**}_{d}$}{}}

Recall that our aim is to determine the expected value $M^{**}_{d}$ of the number of all but simple points in arrangements of $\Omega^{**}$ in order to 
compare it with the value of the complexity $C_d$.

Thanks to the results of the above section and knowing that the first case concerns $3d-3$ points and the second one $d^2-3d+2$ points, we get:
\begin{align*}
M^{**}_{d}=&d^2-\Big(3(d-3)\times\big(\frac{d-4}{d-1}\big)^{\frac{d-5}{6}}+3(d-3)\times\big(\frac{d^2-9d+26}{d^2-3d+2}\big)^{\frac{d-5}{6}}\\
&+(d^2-6d+11)\times\frac{d^2-10d+25}{d^2-3d+2}\big(\frac{d^2-9d+26}{d^2-3d+2}\big)^{\frac{d-11}{6}}\Big).
\end{align*}

Using the Computer Algebra System Xcas, we obtain:
\begin{align*}
M^{**}_{d}=(1-\frac{1}{e})d^2+\big(\frac{1}{e}-3\exp(-\frac{1}{2})\big)d+O(1),\quad\text{as $d\rightarrow\infty$}.
\end{align*}

\subsection{Results}

Figure \ref{Fig1} shows values of both $\frac{C_d-M^*_{d}}{d}$ and $\frac{C_d-M^{**}_{d}}{d}$ for the selected prime numbers $d$.

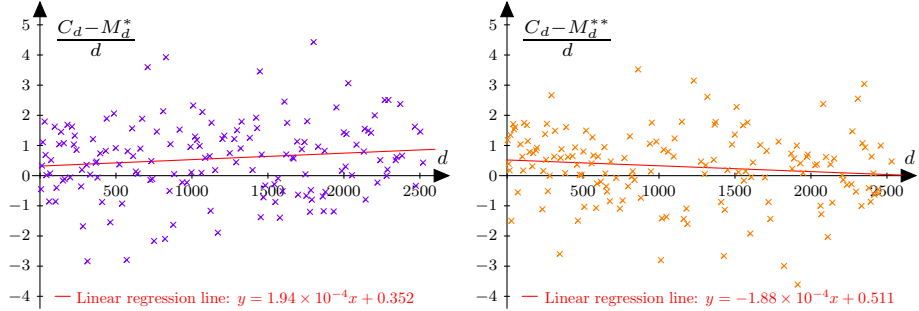
\begin{figure}[ht]
\centering
\definecolor{ffqqqq}{rgb}{1.,0.,0.}
\definecolor{xfqqff}{rgb}{0.4980392156862745,0.,1.}
\begin{tikzpicture}[line cap=round,line join=round,>=triangle 45,x=0.002cm,y=0.4cm]
\draw[->,color=black] (-50,0.) -- (2700,0.);
\foreach \x in {,500,1000,1500,2000,2500}\draw[shift={(\x,0)},color=black,scale=0.5] (0pt,2pt) -- (0pt,-2pt) node[below,scale=0.75] {\footnotesize $\x$};
\draw[->,color=black] (0.,-4.4) -- (0.,5.8);
\foreach \y in {-4,-3,-2,-1,1,2,3,4,5}\draw[shift={(0,\y)},color=black,scale=0.5] (2pt,0pt) -- (-2pt,0pt) node[left,scale=0.75] {\footnotesize $\y$};
\draw[color=black] (-1pt,0pt) node[left,scale=0.75] {\footnotesize $0$};
\clip(-50,-4.4) rectangle (2700,5.8);\draw (30.,5.5) node[anchor=north west] {$\frac{C_d-M^*_{d}}{d}$};\draw (2650,0.6) node[scale=0.75] {$d$};
\draw [color=ffqqqq] (100,-4)-- (200,-4);\draw [color=ffqqqq](1350,-4.05) node[scale=0.6] {Linear regression line: $y=1.94\times10^{-4}x+0.352$};
\draw [color=ffqqqq] (0.,0.3168746394)-- (2600,0.874316492);
\begin{scriptsize}
\draw [color=xfqqff] (7.,-0.4485668012)-- ++(-1.0pt,-1.0pt) -- ++(2.0pt,2.0pt) ++(-2.0pt,0) -- ++(2.0pt,-2.0pt);
\draw [color=xfqqff] (13.,0.3321191375)-- ++(-1.0pt,-1.0pt) -- ++(2.0pt,2.0pt) ++(-2.0pt,0) -- ++(2.0pt,-2.0pt);
\draw [color=xfqqff] (19.,1.1095833276)-- ++(-1.0pt,-1.0pt) -- ++(2.0pt,2.0pt) ++(-2.0pt,0) -- ++(2.0pt,-2.0pt);
\draw [color=xfqqff] (31.,1.7931753739)-- ++(-1.0pt,-1.0pt) -- ++(2.0pt,2.0pt) ++(-2.0pt,0) -- ++(2.0pt,-2.0pt);
\draw [color=xfqqff] (37.,0.6886904275)-- ++(-1.0pt,-1.0pt) -- ++(2.0pt,2.0pt) ++(-2.0pt,0) -- ++(2.0pt,-2.0pt);
\draw [color=xfqqff] (43.,0.019443972)-- ++(-1.0pt,-1.0pt) -- ++(2.0pt,2.0pt) ++(-2.0pt,0) -- ++(2.0pt,-2.0pt);
\draw [color=xfqqff] (61.,-0.856602215)-- ++(-1.0pt,-1.0pt) -- ++(2.0pt,2.0pt) ++(-2.0pt,0) -- ++(2.0pt,-2.0pt);
\draw [color=xfqqff] (67.,0.074805433)-- ++(-1.0pt,-1.0pt) -- ++(2.0pt,2.0pt) ++(-2.0pt,0) -- ++(2.0pt,-2.0pt);
\draw [color=xfqqff] (73.,0.5172610255)-- ++(-1.0pt,-1.0pt) -- ++(2.0pt,2.0pt) ++(-2.0pt,0) -- ++(2.0pt,-2.0pt);
\draw [color=xfqqff] (97.,1.6247558089)-- ++(-1.0pt,-1.0pt) -- ++(2.0pt,2.0pt) ++(-2.0pt,0) -- ++(2.0pt,-2.0pt);
\draw [color=xfqqff] (103.,-0.8565433804)-- ++(-1.0pt,-1.0pt) -- ++(2.0pt,2.0pt) ++(-2.0pt,0) -- ++(2.0pt,-2.0pt);
\draw [color=xfqqff] (109.,-0.3996601877)-- ++(-1.0pt,-1.0pt) -- ++(2.0pt,2.0pt) ++(-2.0pt,0) -- ++(2.0pt,-2.0pt);
\draw [color=xfqqff] (127.,1.0420416694)-- ++(-1.0pt,-1.0pt) -- ++(2.0pt,2.0pt) ++(-2.0pt,0) -- ++(2.0pt,-2.0pt);
\draw [color=xfqqff] (139.,1.4535337012)-- ++(-1.0pt,-1.0pt) -- ++(2.0pt,2.0pt) ++(-2.0pt,0) -- ++(2.0pt,-2.0pt);
\draw [color=xfqqff] (151.,-0.1609886874)-- ++(-1.0pt,-1.0pt) -- ++(2.0pt,2.0pt) ++(-2.0pt,0) -- ++(2.0pt,-2.0pt);
\draw [color=xfqqff] (157.,1.0779996947)-- ++(-1.0pt,-1.0pt) -- ++(2.0pt,2.0pt) ++(-2.0pt,0) -- ++(2.0pt,-2.0pt);
\draw [color=xfqqff] (163.,-0.1147961124)-- ++(-1.0pt,-1.0pt) -- ++(2.0pt,2.0pt) ++(-2.0pt,0) -- ++(2.0pt,-2.0pt);
\draw [color=xfqqff] (181.,1.6869480108)-- ++(-1.0pt,-1.0pt) -- ++(2.0pt,2.0pt) ++(-2.0pt,0) -- ++(2.0pt,-2.0pt);
\draw [color=xfqqff] (199.,-0.0564263037)-- ++(-1.0pt,-1.0pt) -- ++(2.0pt,2.0pt) ++(-2.0pt,0) -- ++(2.0pt,-2.0pt);
\draw [color=xfqqff] (211.,1.6263638511)-- ++(-1.0pt,-1.0pt) -- ++(2.0pt,2.0pt) ++(-2.0pt,0) -- ++(2.0pt,-2.0pt);
\draw [color=xfqqff] (223.,1.033646636)-- ++(-1.0pt,-1.0pt) -- ++(2.0pt,2.0pt) ++(-2.0pt,0) -- ++(2.0pt,-2.0pt);
\draw [color=xfqqff] (229.,1.327092768)-- ++(-1.0pt,-1.0pt) -- ++(2.0pt,2.0pt) ++(-2.0pt,0) -- ++(2.0pt,-2.0pt);
\draw [color=xfqqff] (241.,0.8554690058)-- ++(-1.0pt,-1.0pt) -- ++(2.0pt,2.0pt) ++(-2.0pt,0) -- ++(2.0pt,-2.0pt);
\draw [color=xfqqff] (271.,-0.3566260386)-- ++(-1.0pt,-1.0pt) -- ++(2.0pt,2.0pt) ++(-2.0pt,0) -- ++(2.0pt,-2.0pt);
\draw [color=xfqqff] (277.,0.1936945446)-- ++(-1.0pt,-1.0pt) -- ++(2.0pt,2.0pt) ++(-2.0pt,0) -- ++(2.0pt,-2.0pt);
\draw [color=xfqqff] (283.,-1.6768915058)-- ++(-1.0pt,-1.0pt) -- ++(2.0pt,2.0pt) ++(-2.0pt,0) -- ++(2.0pt,-2.0pt);
\draw [color=xfqqff] (307.,0.359797135)-- ++(-1.0pt,-1.0pt) -- ++(2.0pt,2.0pt) ++(-2.0pt,0) -- ++(2.0pt,-2.0pt);
\draw [color=xfqqff] (313.,-2.8380817044)-- ++(-1.0pt,-1.0pt) -- ++(2.0pt,2.0pt) ++(-2.0pt,0) -- ++(2.0pt,-2.0pt);
\draw [color=xfqqff] (331.,0.0370763252)-- ++(-1.0pt,-1.0pt) -- ++(2.0pt,2.0pt) ++(-2.0pt,0) -- ++(2.0pt,-2.0pt);
\draw [color=xfqqff] (349.,1.2098845714)-- ++(-1.0pt,-1.0pt) -- ++(2.0pt,2.0pt) ++(-2.0pt,0) -- ++(2.0pt,-2.0pt);
\draw [color=xfqqff] (367.,0.4321875229)-- ++(-1.0pt,-1.0pt) -- ++(2.0pt,2.0pt) ++(-2.0pt,0) -- ++(2.0pt,-2.0pt);
\draw [color=xfqqff] (373.,0.4683872693)-- ++(-1.0pt,-1.0pt) -- ++(2.0pt,2.0pt) ++(-2.0pt,0) -- ++(2.0pt,-2.0pt);
\draw [color=xfqqff] (379.,-0.0599171994)-- ++(-1.0pt,-1.0pt) -- ++(2.0pt,2.0pt) ++(-2.0pt,0) -- ++(2.0pt,-2.0pt);
\draw [color=xfqqff] (397.,0.740463419)-- ++(-1.0pt,-1.0pt) -- ++(2.0pt,2.0pt) ++(-2.0pt,0) -- ++(2.0pt,-2.0pt);
\draw [color=xfqqff] (409.,-0.8740450988)-- ++(-1.0pt,-1.0pt) -- ++(2.0pt,2.0pt) ++(-2.0pt,0) -- ++(2.0pt,-2.0pt);
\draw [color=xfqqff] (433.,0.8209384913)-- ++(-1.0pt,-1.0pt) -- ++(2.0pt,2.0pt) ++(-2.0pt,0) -- ++(2.0pt,-2.0pt);
\draw [color=xfqqff] (439.,1.8927615841)-- ++(-1.0pt,-1.0pt) -- ++(2.0pt,2.0pt) ++(-2.0pt,0) -- ++(2.0pt,-2.0pt);
\draw [color=xfqqff] (463.,-1.5472585563)-- ++(-1.0pt,-1.0pt) -- ++(2.0pt,2.0pt) ++(-2.0pt,0) -- ++(2.0pt,-2.0pt);
\draw [color=xfqqff] (487.,2.0618340617)-- ++(-1.0pt,-1.0pt) -- ++(2.0pt,2.0pt) ++(-2.0pt,0) -- ++(2.0pt,-2.0pt);
\draw [color=xfqqff] (499.,0.9164008028)-- ++(-1.0pt,-1.0pt) -- ++(2.0pt,2.0pt) ++(-2.0pt,0) -- ++(2.0pt,-2.0pt);
\draw [color=xfqqff] (523.,0.371841996)-- ++(-1.0pt,-1.0pt) -- ++(2.0pt,2.0pt) ++(-2.0pt,0) -- ++(2.0pt,-2.0pt);
\draw [color=xfqqff] (541.,-0.9256337041)-- ++(-1.0pt,-1.0pt) -- ++(2.0pt,2.0pt) ++(-2.0pt,0) -- ++(2.0pt,-2.0pt);
\draw [color=xfqqff] (571.,-2.7929192607)-- ++(-1.0pt,-1.0pt) -- ++(2.0pt,2.0pt) ++(-2.0pt,0) -- ++(2.0pt,-2.0pt);
\draw [color=xfqqff] (577.,0.8109310764)-- ++(-1.0pt,-1.0pt) -- ++(2.0pt,2.0pt) ++(-2.0pt,0) -- ++(2.0pt,-2.0pt);
\draw [color=xfqqff] (607.,1.5676072385)-- ++(-1.0pt,-1.0pt) -- ++(2.0pt,2.0pt) ++(-2.0pt,0) -- ++(2.0pt,-2.0pt);
\draw [color=xfqqff] (613.,0.6494708173)-- ++(-1.0pt,-1.0pt) -- ++(2.0pt,2.0pt) ++(-2.0pt,0) -- ++(2.0pt,-2.0pt);
\draw [color=xfqqff] (631.,0.1344113385)-- ++(-1.0pt,-1.0pt) -- ++(2.0pt,2.0pt) ++(-2.0pt,0) -- ++(2.0pt,-2.0pt);
\draw [color=xfqqff] (643.,0.0690131076)-- ++(-1.0pt,-1.0pt) -- ++(2.0pt,2.0pt) ++(-2.0pt,0) -- ++(2.0pt,-2.0pt);
\draw [color=xfqqff] (661.,0.8976900267)-- ++(-1.0pt,-1.0pt) -- ++(2.0pt,2.0pt) ++(-2.0pt,0) -- ++(2.0pt,-2.0pt);
\draw [color=xfqqff] (673.,1.8882736125)-- ++(-1.0pt,-1.0pt) -- ++(2.0pt,2.0pt) ++(-2.0pt,0) -- ++(2.0pt,-2.0pt);
\draw [color=xfqqff] (709.,3.5963453271)-- ++(-1.0pt,-1.0pt) -- ++(2.0pt,2.0pt) ++(-2.0pt,0) -- ++(2.0pt,-2.0pt);
\draw [color=xfqqff] (727.,-0.1432077565)-- ++(-1.0pt,-1.0pt) -- ++(2.0pt,2.0pt) ++(-2.0pt,0) -- ++(2.0pt,-2.0pt);
\draw [color=xfqqff] (733.,0.4846143289)-- ++(-1.0pt,-1.0pt) -- ++(2.0pt,2.0pt) ++(-2.0pt,0) -- ++(2.0pt,-2.0pt);
\draw [color=xfqqff] (739.,-0.1122544489)-- ++(-1.0pt,-1.0pt) -- ++(2.0pt,2.0pt) ++(-2.0pt,0) -- ++(2.0pt,-2.0pt);
\draw [color=xfqqff] (751.,-2.166246029)-- ++(-1.0pt,-1.0pt) -- ++(2.0pt,2.0pt) ++(-2.0pt,0) -- ++(2.0pt,-2.0pt);
\draw [color=xfqqff] (769.,1.4609389039)-- ++(-1.0pt,-1.0pt) -- ++(2.0pt,2.0pt) ++(-2.0pt,0) -- ++(2.0pt,-2.0pt);
\draw [color=xfqqff] (811.,2.1355297524)-- ++(-1.0pt,-1.0pt) -- ++(2.0pt,2.0pt) ++(-2.0pt,0) -- ++(2.0pt,-2.0pt);
\draw [color=xfqqff] (823.,-2.0975510075)-- ++(-1.0pt,-1.0pt) -- ++(2.0pt,2.0pt) ++(-2.0pt,0) -- ++(2.0pt,-2.0pt);
\draw [color=xfqqff] (829.,3.9262327425)-- ++(-1.0pt,-1.0pt) -- ++(2.0pt,2.0pt) ++(-2.0pt,0) -- ++(2.0pt,-2.0pt);
\draw [color=xfqqff] (853.,-0.5038149896)-- ++(-1.0pt,-1.0pt) -- ++(2.0pt,2.0pt) ++(-2.0pt,0) -- ++(2.0pt,-2.0pt);
\draw [color=xfqqff] (859.,1.0451848971)-- ++(-1.0pt,-1.0pt) -- ++(2.0pt,2.0pt) ++(-2.0pt,0) -- ++(2.0pt,-2.0pt);
\draw [color=xfqqff] (877.,-1.6289147392)-- ++(-1.0pt,-1.0pt) -- ++(2.0pt,2.0pt) ++(-2.0pt,0) -- ++(2.0pt,-2.0pt);
\draw [color=xfqqff] (883.,0.0497837589)-- ++(-1.0pt,-1.0pt) -- ++(2.0pt,2.0pt) ++(-2.0pt,0) -- ++(2.0pt,-2.0pt);
\draw [color=xfqqff] (919.,1.5280701372)-- ++(-1.0pt,-1.0pt) -- ++(2.0pt,2.0pt) ++(-2.0pt,0) -- ++(2.0pt,-2.0pt);
\draw [color=xfqqff] (937.,-0.2224717684)-- ++(-1.0pt,-1.0pt) -- ++(2.0pt,2.0pt) ++(-2.0pt,0) -- ++(2.0pt,-2.0pt);
\draw [color=xfqqff] (967.,-0.2281597883)-- ++(-1.0pt,-1.0pt) -- ++(2.0pt,2.0pt) ++(-2.0pt,0) -- ++(2.0pt,-2.0pt);
\draw [color=xfqqff] (991.,0.9582256762)-- ++(-1.0pt,-1.0pt) -- ++(2.0pt,2.0pt) ++(-2.0pt,0) -- ++(2.0pt,-2.0pt);
\draw [color=xfqqff] (997.,0.1087040942)-- ++(-1.0pt,-1.0pt) -- ++(2.0pt,2.0pt) ++(-2.0pt,0) -- ++(2.0pt,-2.0pt);
\draw [color=xfqqff] (1009.,2.3274456012)-- ++(-1.0pt,-1.0pt) -- ++(2.0pt,2.0pt) ++(-2.0pt,0) -- ++(2.0pt,-2.0pt);
\draw [color=xfqqff] (1021.,1.2951581943)-- ++(-1.0pt,-1.0pt) -- ++(2.0pt,2.0pt) ++(-2.0pt,0) -- ++(2.0pt,-2.0pt);
\draw [color=xfqqff] (1033.,1.1822552723)-- ++(-1.0pt,-1.0pt) -- ++(2.0pt,2.0pt) ++(-2.0pt,0) -- ++(2.0pt,-2.0pt);
\draw [color=xfqqff] (1051.,0.9841050023)-- ++(-1.0pt,-1.0pt) -- ++(2.0pt,2.0pt) ++(-2.0pt,0) -- ++(2.0pt,-2.0pt);
\draw [color=xfqqff] (1063.,0.0941493187)-- ++(-1.0pt,-1.0pt) -- ++(2.0pt,2.0pt) ++(-2.0pt,0) -- ++(2.0pt,-2.0pt);
\draw [color=xfqqff] (1069.,2.1129138338)-- ++(-1.0pt,-1.0pt) -- ++(2.0pt,2.0pt) ++(-2.0pt,0) -- ++(2.0pt,-2.0pt);
\draw [color=xfqqff] (1087.,0.5919975065)-- ++(-1.0pt,-1.0pt) -- ++(2.0pt,2.0pt) ++(-2.0pt,0) -- ++(2.0pt,-2.0pt);
\draw [color=xfqqff] (1117.,0.6474858288)-- ++(-1.0pt,-1.0pt) -- ++(2.0pt,2.0pt) ++(-2.0pt,0) -- ++(2.0pt,-2.0pt);
\draw [color=xfqqff] (1123.,-1.1645131754)-- ++(-1.0pt,-1.0pt) -- ++(2.0pt,2.0pt) ++(-2.0pt,0) -- ++(2.0pt,-2.0pt);
\draw [color=xfqqff] (1129.,1.7561993413)-- ++(-1.0pt,-1.0pt) -- ++(2.0pt,2.0pt) ++(-2.0pt,0) -- ++(2.0pt,-2.0pt);
\draw [color=xfqqff] (1153.,0.1835313582)-- ++(-1.0pt,-1.0pt) -- ++(2.0pt,2.0pt) ++(-2.0pt,0) -- ++(2.0pt,-2.0pt);
\draw [color=xfqqff] (1171.,-1.8907980733)-- ++(-1.0pt,-1.0pt) -- ++(2.0pt,2.0pt) ++(-2.0pt,0) -- ++(2.0pt,-2.0pt);
\draw [color=xfqqff] (1201.,1.2179252059)-- ++(-1.0pt,-1.0pt) -- ++(2.0pt,2.0pt) ++(-2.0pt,0) -- ++(2.0pt,-2.0pt);
\draw [color=xfqqff] (1213.,1.3247159875)-- ++(-1.0pt,-1.0pt) -- ++(2.0pt,2.0pt) ++(-2.0pt,0) -- ++(2.0pt,-2.0pt);
\draw [color=xfqqff] (1249.,0.7486253925)-- ++(-1.0pt,-1.0pt) -- ++(2.0pt,2.0pt) ++(-2.0pt,0) -- ++(2.0pt,-2.0pt);
\draw [color=xfqqff] (1279.,1.2177963873)-- ++(-1.0pt,-1.0pt) -- ++(2.0pt,2.0pt) ++(-2.0pt,0) -- ++(2.0pt,-2.0pt);
\draw [color=xfqqff] (1291.,0.4456887682)-- ++(-1.0pt,-1.0pt) -- ++(2.0pt,2.0pt) ++(-2.0pt,0) -- ++(2.0pt,-2.0pt);
\draw [color=xfqqff] (1297.,1.5077305235)-- ++(-1.0pt,-1.0pt) -- ++(2.0pt,2.0pt) ++(-2.0pt,0) -- ++(2.0pt,-2.0pt);
\draw [color=xfqqff] (1321.,0.9045941245)-- ++(-1.0pt,-1.0pt) -- ++(2.0pt,2.0pt) ++(-2.0pt,0) -- ++(2.0pt,-2.0pt);
\draw [color=xfqqff] (1327.,1.7916712315)-- ++(-1.0pt,-1.0pt) -- ++(2.0pt,2.0pt) ++(-2.0pt,0) -- ++(2.0pt,-2.0pt);
\draw [color=xfqqff] (1381.,1.2524888236)-- ++(-1.0pt,-1.0pt) -- ++(2.0pt,2.0pt) ++(-2.0pt,0) -- ++(2.0pt,-2.0pt);
\draw [color=xfqqff] (1399.,0.1876377848)-- ++(-1.0pt,-1.0pt) -- ++(2.0pt,2.0pt) ++(-2.0pt,0) -- ++(2.0pt,-2.0pt);
\draw [color=xfqqff] (1423.,1.9234612273)-- ++(-1.0pt,-1.0pt) -- ++(2.0pt,2.0pt) ++(-2.0pt,0) -- ++(2.0pt,-2.0pt);
\draw [color=xfqqff] (1429.,1.5779906517)-- ++(-1.0pt,-1.0pt) -- ++(2.0pt,2.0pt) ++(-2.0pt,0) -- ++(2.0pt,-2.0pt);
\draw [color=xfqqff] (1447.,3.4571268157)-- ++(-1.0pt,-1.0pt) -- ++(2.0pt,2.0pt) ++(-2.0pt,0) -- ++(2.0pt,-2.0pt);
\draw [color=xfqqff] (1453.,-1.3769613937)-- ++(-1.0pt,-1.0pt) -- ++(2.0pt,2.0pt) ++(-2.0pt,0) -- ++(2.0pt,-2.0pt);
\draw [color=xfqqff] (1459.,0.70018849)-- ++(-1.0pt,-1.0pt) -- ++(2.0pt,2.0pt) ++(-2.0pt,0) -- ++(2.0pt,-2.0pt);
\draw [color=xfqqff] (1471.,-0.7181259081)-- ++(-1.0pt,-1.0pt) -- ++(2.0pt,2.0pt) ++(-2.0pt,0) -- ++(2.0pt,-2.0pt);
\draw [color=xfqqff] (1483.,-0.3245800074)-- ++(-1.0pt,-1.0pt) -- ++(2.0pt,2.0pt) ++(-2.0pt,0) -- ++(2.0pt,-2.0pt);
\draw [color=xfqqff] (1543.,-0.2864923419)-- ++(-1.0pt,-1.0pt) -- ++(2.0pt,2.0pt) ++(-2.0pt,0) -- ++(2.0pt,-2.0pt);
\draw [color=xfqqff] (1549.,-0.8705280695)-- ++(-1.0pt,-1.0pt) -- ++(2.0pt,2.0pt) ++(-2.0pt,0) -- ++(2.0pt,-2.0pt);
\draw [color=xfqqff] (1567.,-0.0897861846)-- ++(-1.0pt,-1.0pt) -- ++(2.0pt,2.0pt) ++(-2.0pt,0) -- ++(2.0pt,-2.0pt);
\draw [color=xfqqff] (1579.,-1.3939769906)-- ++(-1.0pt,-1.0pt) -- ++(2.0pt,2.0pt) ++(-2.0pt,0) -- ++(2.0pt,-2.0pt);
\draw [color=xfqqff] (1597.,-0.5122272056)-- ++(-1.0pt,-1.0pt) -- ++(2.0pt,2.0pt) ++(-2.0pt,0) -- ++(2.0pt,-2.0pt);
\draw [color=xfqqff] (1609.,2.4552216475)-- ++(-1.0pt,-1.0pt) -- ++(2.0pt,2.0pt) ++(-2.0pt,0) -- ++(2.0pt,-2.0pt);
\draw [color=xfqqff] (1621.,0.7044289865)-- ++(-1.0pt,-1.0pt) -- ++(2.0pt,2.0pt) ++(-2.0pt,0) -- ++(2.0pt,-2.0pt);
\draw [color=xfqqff] (1627.,1.7521968313)-- ++(-1.0pt,-1.0pt) -- ++(2.0pt,2.0pt) ++(-2.0pt,0) -- ++(2.0pt,-2.0pt);
\draw [color=xfqqff] (1663.,0.6017601316)-- ++(-1.0pt,-1.0pt) -- ++(2.0pt,2.0pt) ++(-2.0pt,0) -- ++(2.0pt,-2.0pt);
\draw [color=xfqqff] (1669.,0.5698521228)-- ++(-1.0pt,-1.0pt) -- ++(2.0pt,2.0pt) ++(-2.0pt,0) -- ++(2.0pt,-2.0pt);
\draw [color=xfqqff] (1693.,1.1325344152)-- ++(-1.0pt,-1.0pt) -- ++(2.0pt,2.0pt) ++(-2.0pt,0) -- ++(2.0pt,-2.0pt);
\draw [color=xfqqff] (1699.,-0.6308572944)-- ++(-1.0pt,-1.0pt) -- ++(2.0pt,2.0pt) ++(-2.0pt,0) -- ++(2.0pt,-2.0pt);
\draw [color=xfqqff] (1723.,0.8765045965)-- ++(-1.0pt,-1.0pt) -- ++(2.0pt,2.0pt) ++(-2.0pt,0) -- ++(2.0pt,-2.0pt);
\draw [color=xfqqff] (1741.,1.1133643006)-- ++(-1.0pt,-1.0pt) -- ++(2.0pt,2.0pt) ++(-2.0pt,0) -- ++(2.0pt,-2.0pt);
\draw [color=xfqqff] (1747.,0.0452538613)-- ++(-1.0pt,-1.0pt) -- ++(2.0pt,2.0pt) ++(-2.0pt,0) -- ++(2.0pt,-2.0pt);
\draw [color=xfqqff] (1753.,-0.6650102318)-- ++(-1.0pt,-1.0pt) -- ++(2.0pt,2.0pt) ++(-2.0pt,0) -- ++(2.0pt,-2.0pt);
\draw [color=xfqqff] (1759.,1.8049476697)-- ++(-1.0pt,-1.0pt) -- ++(2.0pt,2.0pt) ++(-2.0pt,0) -- ++(2.0pt,-2.0pt);
\draw [color=xfqqff] (1777.,-1.2004363923)-- ++(-1.0pt,-1.0pt) -- ++(2.0pt,2.0pt) ++(-2.0pt,0) -- ++(2.0pt,-2.0pt);
\draw [color=xfqqff] (1783.,-0.4809648364)-- ++(-1.0pt,-1.0pt) -- ++(2.0pt,2.0pt) ++(-2.0pt,0) -- ++(2.0pt,-2.0pt);
\draw [color=xfqqff] (1789.,-0.7962313124)-- ++(-1.0pt,-1.0pt) -- ++(2.0pt,2.0pt) ++(-2.0pt,0) -- ++(2.0pt,-2.0pt);
\draw [color=xfqqff] (1801.,4.4271011169)-- ++(-1.0pt,-1.0pt) -- ++(2.0pt,2.0pt) ++(-2.0pt,0) -- ++(2.0pt,-2.0pt);
\draw [color=xfqqff] (1831.,0.9628096822)-- ++(-1.0pt,-1.0pt) -- ++(2.0pt,2.0pt) ++(-2.0pt,0) -- ++(2.0pt,-2.0pt);
\draw [color=xfqqff] (1861.,-0.6863068594)-- ++(-1.0pt,-1.0pt) -- ++(2.0pt,2.0pt) ++(-2.0pt,0) -- ++(2.0pt,-2.0pt);
\draw [color=xfqqff] (1867.,-1.1777072444)-- ++(-1.0pt,-1.0pt) -- ++(2.0pt,2.0pt) ++(-2.0pt,0) -- ++(2.0pt,-2.0pt);
\draw [color=xfqqff] (1873.,0.8244236703)-- ++(-1.0pt,-1.0pt) -- ++(2.0pt,2.0pt) ++(-2.0pt,0) -- ++(2.0pt,-2.0pt);
\draw [color=xfqqff] (1879.,-0.540948476)-- ++(-1.0pt,-1.0pt) -- ++(2.0pt,2.0pt) ++(-2.0pt,0) -- ++(2.0pt,-2.0pt);
\draw [color=xfqqff] (1933.,-1.1861367253)-- ++(-1.0pt,-1.0pt) -- ++(2.0pt,2.0pt) ++(-2.0pt,0) -- ++(2.0pt,-2.0pt);
\draw [color=xfqqff] (1951.,2.2626940519)-- ++(-1.0pt,-1.0pt) -- ++(2.0pt,2.0pt) ++(-2.0pt,0) -- ++(2.0pt,-2.0pt);
\draw [color=xfqqff] (1987.,0.1561703551)-- ++(-1.0pt,-1.0pt) -- ++(2.0pt,2.0pt) ++(-2.0pt,0) -- ++(2.0pt,-2.0pt);
\draw [color=xfqqff] (1999.,1.4357881538)-- ++(-1.0pt,-1.0pt) -- ++(2.0pt,2.0pt) ++(-2.0pt,0) -- ++(2.0pt,-2.0pt);
\draw [color=xfqqff] (2011.,0.3003082158)-- ++(-1.0pt,-1.0pt) -- ++(2.0pt,2.0pt) ++(-2.0pt,0) -- ++(2.0pt,-2.0pt);
\draw [color=xfqqff] (2017.,1.3815272447)-- ++(-1.0pt,-1.0pt) -- ++(2.0pt,2.0pt) ++(-2.0pt,0) -- ++(2.0pt,-2.0pt);
\draw [color=xfqqff] (2029.,3.0641910659)-- ++(-1.0pt,-1.0pt) -- ++(2.0pt,2.0pt) ++(-2.0pt,0) -- ++(2.0pt,-2.0pt);
\draw [color=xfqqff] (2053.,0.0138811246)-- ++(-1.0pt,-1.0pt) -- ++(2.0pt,2.0pt) ++(-2.0pt,0) -- ++(2.0pt,-2.0pt);
\draw [color=xfqqff] (2083.,1.0290592828)-- ++(-1.0pt,-1.0pt) -- ++(2.0pt,2.0pt) ++(-2.0pt,0) -- ++(2.0pt,-2.0pt);
\draw [color=xfqqff] (2131.,0.8005270612)-- ++(-1.0pt,-1.0pt) -- ++(2.0pt,2.0pt) ++(-2.0pt,0) -- ++(2.0pt,-2.0pt);
\draw [color=xfqqff] (2137.,1.5537155784)-- ++(-1.0pt,-1.0pt) -- ++(2.0pt,2.0pt) ++(-2.0pt,0) -- ++(2.0pt,-2.0pt);
\draw [color=xfqqff] (2143.,-0.7913464609)-- ++(-1.0pt,-1.0pt) -- ++(2.0pt,2.0pt) ++(-2.0pt,0) -- ++(2.0pt,-2.0pt);
\draw [color=xfqqff] (2161.,1.4931196122)-- ++(-1.0pt,-1.0pt) -- ++(2.0pt,2.0pt) ++(-2.0pt,0) -- ++(2.0pt,-2.0pt);
\draw [color=xfqqff] (2179.,1.4219845373)-- ++(-1.0pt,-1.0pt) -- ++(2.0pt,2.0pt) ++(-2.0pt,0) -- ++(2.0pt,-2.0pt);
\draw [color=xfqqff] (2203.,1.9988608935)-- ++(-1.0pt,-1.0pt) -- ++(2.0pt,2.0pt) ++(-2.0pt,0) -- ++(2.0pt,-2.0pt);
\draw [color=xfqqff] (2221.,-0.2289072102)-- ++(-1.0pt,-1.0pt) -- ++(2.0pt,2.0pt) ++(-2.0pt,0) -- ++(2.0pt,-2.0pt);
\draw [color=xfqqff] (2239.,0.5154486382)-- ++(-1.0pt,-1.0pt) -- ++(2.0pt,2.0pt) ++(-2.0pt,0) -- ++(2.0pt,-2.0pt);
\draw [color=xfqqff] (2269.,2.5057984071)-- ++(-1.0pt,-1.0pt) -- ++(2.0pt,2.0pt) ++(-2.0pt,0) -- ++(2.0pt,-2.0pt);
\draw [color=xfqqff] (2281.,-0.0020628421)-- ++(-1.0pt,-1.0pt) -- ++(2.0pt,2.0pt) ++(-2.0pt,0) -- ++(2.0pt,-2.0pt);
\draw [color=xfqqff] (2287.,0.2023807051)-- ++(-1.0pt,-1.0pt) -- ++(2.0pt,2.0pt) ++(-2.0pt,0) -- ++(2.0pt,-2.0pt);
\draw [color=xfqqff] (2293.,2.5106332349)-- ++(-1.0pt,-1.0pt) -- ++(2.0pt,2.0pt) ++(-2.0pt,0) -- ++(2.0pt,-2.0pt);
\draw [color=xfqqff] (2341.,0.5141923093)-- ++(-1.0pt,-1.0pt) -- ++(2.0pt,2.0pt) ++(-2.0pt,0) -- ++(2.0pt,-2.0pt);
\draw [color=xfqqff] (2347.,0.6165945497)-- ++(-1.0pt,-1.0pt) -- ++(2.0pt,2.0pt) ++(-2.0pt,0) -- ++(2.0pt,-2.0pt);
\draw [color=xfqqff] (2371.,2.3752481879)-- ++(-1.0pt,-1.0pt) -- ++(2.0pt,2.0pt) ++(-2.0pt,0) -- ++(2.0pt,-2.0pt);
\draw [color=xfqqff] (2383.,0.6617374714)-- ++(-1.0pt,-1.0pt) -- ++(2.0pt,2.0pt) ++(-2.0pt,0) -- ++(2.0pt,-2.0pt);
\draw [color=xfqqff] (2467.,1.625029909)-- ++(-1.0pt,-1.0pt) -- ++(2.0pt,2.0pt) ++(-2.0pt,0) -- ++(2.0pt,-2.0pt);
\draw [color=xfqqff] (2473.,-0.3386147549)-- ++(-1.0pt,-1.0pt) -- ++(2.0pt,2.0pt) ++(-2.0pt,0) -- ++(2.0pt,-2.0pt);
\draw [color=xfqqff] (2503.,1.4599281401)-- ++(-1.0pt,-1.0pt) -- ++(2.0pt,2.0pt) ++(-2.0pt,0) -- ++(2.0pt,-2.0pt);
\draw [color=xfqqff] (2521.,0.4321319418)-- ++(-1.0pt,-1.0pt) -- ++(2.0pt,2.0pt) ++(-2.0pt,0) -- ++(2.0pt,-2.0pt);
\end{scriptsize}
\end{tikzpicture}
\definecolor{ffxfqq}{rgb}{1.,0.4980392156862745,0.}
\begin{tikzpicture}[line cap=round,line join=round,>=triangle 45,x=0.002cm,y=0.4cm]
\draw[->,color=black] (-50,0.) -- (2700,0.);
\foreach \x in {,500,1000,1500,2000,2500}\draw[shift={(\x,0)},color=black,scale=0.5] (0pt,2pt) -- (0pt,-2pt) node[below,scale=0.75] {\footnotesize $\x$};
\draw[->,color=black] (0.,-4.4) -- (0.,5.8);
\foreach \y in {-4,-3,-2,-1,1,2,3,4,5}\draw[shift={(0,\y)},color=black,scale=0.5] (2pt,0pt) -- (-2pt,0pt) node[left,scale=0.75] {\footnotesize $\y$};
\draw[color=black] (-1pt,0pt) node[left,scale=0.75] {\footnotesize $0$};\clip(-50,-4.4) rectangle (2700,5.8);\draw [color=ffqqqq] (100,-4)-- (200,-4);
\draw [color=ffqqqq](1400,-4.05) node[scale=0.6] {Linear regression line: $y=-1.88\times10^{-4}x+0.511$};\draw (2650,0.6) node[scale=0.75] {$d$};
\draw [color=ffqqqq] (0.,0.5199578177)-- (2600,0.0183440266);\draw (30.,5.5) node[anchor=north west] {$\frac{C_d-M^{**}_{d}}{d}$};
\begin{scriptsize}
\draw [color=ffxfqq] (5.,1.1755266274)-- ++(-1.0pt,-1.0pt) -- ++(2.0pt,2.0pt) ++(-2.0pt,0) -- ++(2.0pt,-2.0pt);
\draw [color=ffxfqq] (11.,0.3957636219)-- ++(-1.0pt,-1.0pt) -- ++(2.0pt,2.0pt) ++(-2.0pt,0) -- ++(2.0pt,-2.0pt);
\draw [color=ffxfqq] (17.,0.1749073665)-- ++(-1.0pt,-1.0pt) -- ++(2.0pt,2.0pt) ++(-2.0pt,0) -- ++(2.0pt,-2.0pt);
\draw [color=ffxfqq] (23.,1.3507849275)-- ++(-1.0pt,-1.0pt) -- ++(2.0pt,2.0pt) ++(-2.0pt,0) -- ++(2.0pt,-2.0pt);
\draw [color=ffxfqq] (29.,-1.4961104443)-- ++(-1.0pt,-1.0pt) -- ++(2.0pt,2.0pt) ++(-2.0pt,0) -- ++(2.0pt,-2.0pt);
\draw [color=ffxfqq] (41.,1.7100835432)-- ++(-1.0pt,-1.0pt) -- ++(2.0pt,2.0pt) ++(-2.0pt,0) -- ++(2.0pt,-2.0pt);
\draw [color=ffxfqq] (47.,1.5336681548)-- ++(-1.0pt,-1.0pt) -- ++(2.0pt,2.0pt) ++(-2.0pt,0) -- ++(2.0pt,-2.0pt);
\draw [color=ffxfqq] (53.,1.6138551774)-- ++(-1.0pt,-1.0pt) -- ++(2.0pt,2.0pt) ++(-2.0pt,0) -- ++(2.0pt,-2.0pt);
\draw [color=ffxfqq] (71.,-0.5521237295)-- ++(-1.0pt,-1.0pt) -- ++(2.0pt,2.0pt) ++(-2.0pt,0) -- ++(2.0pt,-2.0pt);
\draw [color=ffxfqq] (89.,1.0386358709)-- ++(-1.0pt,-1.0pt) -- ++(2.0pt,2.0pt) ++(-2.0pt,0) -- ++(2.0pt,-2.0pt);
\draw [color=ffxfqq] (101.,0.630015312)-- ++(-1.0pt,-1.0pt) -- ++(2.0pt,2.0pt) ++(-2.0pt,0) -- ++(2.0pt,-2.0pt);
\draw [color=ffxfqq] (107.,0.1725081337)-- ++(-1.0pt,-1.0pt) -- ++(2.0pt,2.0pt) ++(-2.0pt,0) -- ++(2.0pt,-2.0pt);
\draw [color=ffxfqq] (113.,1.750196321)-- ++(-1.0pt,-1.0pt) -- ++(2.0pt,2.0pt) ++(-2.0pt,0) -- ++(2.0pt,-2.0pt);
\draw [color=ffxfqq] (131.,1.6766109326)-- ++(-1.0pt,-1.0pt) -- ++(2.0pt,2.0pt) ++(-2.0pt,0) -- ++(2.0pt,-2.0pt);
\draw [color=ffxfqq] (137.,1.1452425552)-- ++(-1.0pt,-1.0pt) -- ++(2.0pt,2.0pt) ++(-2.0pt,0) -- ++(2.0pt,-2.0pt);
\draw [color=ffxfqq] (149.,0.7844573413)-- ++(-1.0pt,-1.0pt) -- ++(2.0pt,2.0pt) ++(-2.0pt,0) -- ++(2.0pt,-2.0pt);
\draw [color=ffxfqq] (167.,0.3084853446)-- ++(-1.0pt,-1.0pt) -- ++(2.0pt,2.0pt) ++(-2.0pt,0) -- ++(2.0pt,-2.0pt);
\draw [color=ffxfqq] (173.,0.7497253237)-- ++(-1.0pt,-1.0pt) -- ++(2.0pt,2.0pt) ++(-2.0pt,0) -- ++(2.0pt,-2.0pt);
\draw [color=ffxfqq] (191.,0.8386505995)-- ++(-1.0pt,-1.0pt) -- ++(2.0pt,2.0pt) ++(-2.0pt,0) -- ++(2.0pt,-2.0pt);
\draw [color=ffxfqq] (197.,0.9254655507)-- ++(-1.0pt,-1.0pt) -- ++(2.0pt,2.0pt) ++(-2.0pt,0) -- ++(2.0pt,-2.0pt);
\draw [color=ffxfqq] (233.,-1.1143510536)-- ++(-1.0pt,-1.0pt) -- ++(2.0pt,2.0pt) ++(-2.0pt,0) -- ++(2.0pt,-2.0pt);
\draw [color=ffxfqq] (239.,0.1585841719)-- ++(-1.0pt,-1.0pt) -- ++(2.0pt,2.0pt) ++(-2.0pt,0) -- ++(2.0pt,-2.0pt);
\draw [color=ffxfqq] (251.,1.3802960955)-- ++(-1.0pt,-1.0pt) -- ++(2.0pt,2.0pt) ++(-2.0pt,0) -- ++(2.0pt,-2.0pt);
\draw [color=ffxfqq] (257.,0.5270879928)-- ++(-1.0pt,-1.0pt) -- ++(2.0pt,2.0pt) ++(-2.0pt,0) -- ++(2.0pt,-2.0pt);
\draw [color=ffxfqq] (263.,1.7184646606)-- ++(-1.0pt,-1.0pt) -- ++(2.0pt,2.0pt) ++(-2.0pt,0) -- ++(2.0pt,-2.0pt);
\draw [color=ffxfqq] (269.,-0.5132412777)-- ++(-1.0pt,-1.0pt) -- ++(2.0pt,2.0pt) ++(-2.0pt,0) -- ++(2.0pt,-2.0pt);
\draw [color=ffxfqq] (281.,0.0012945895)-- ++(-1.0pt,-1.0pt) -- ++(2.0pt,2.0pt) ++(-2.0pt,0) -- ++(2.0pt,-2.0pt);
\draw [color=ffxfqq] (293.,2.6680500815)-- ++(-1.0pt,-1.0pt) -- ++(2.0pt,2.0pt) ++(-2.0pt,0) -- ++(2.0pt,-2.0pt);
\draw [color=ffxfqq] (311.,0.6316906852)-- ++(-1.0pt,-1.0pt) -- ++(2.0pt,2.0pt) ++(-2.0pt,0) -- ++(2.0pt,-2.0pt);
\draw [color=ffxfqq] (317.,1.490019326)-- ++(-1.0pt,-1.0pt) -- ++(2.0pt,2.0pt) ++(-2.0pt,0) -- ++(2.0pt,-2.0pt);
\draw [color=ffxfqq] (347.,-2.5964052611)-- ++(-1.0pt,-1.0pt) -- ++(2.0pt,2.0pt) ++(-2.0pt,0) -- ++(2.0pt,-2.0pt);
\draw [color=ffxfqq] (353.,0.5916474627)-- ++(-1.0pt,-1.0pt) -- ++(2.0pt,2.0pt) ++(-2.0pt,0) -- ++(2.0pt,-2.0pt);
\draw [color=ffxfqq] (359.,0.8806212968)-- ++(-1.0pt,-1.0pt) -- ++(2.0pt,2.0pt) ++(-2.0pt,0) -- ++(2.0pt,-2.0pt);
\draw [color=ffxfqq] (383.,0.0840264516)-- ++(-1.0pt,-1.0pt) -- ++(2.0pt,2.0pt) ++(-2.0pt,0) -- ++(2.0pt,-2.0pt);
\draw [color=ffxfqq] (389.,1.60902361)-- ++(-1.0pt,-1.0pt) -- ++(2.0pt,2.0pt) ++(-2.0pt,0) -- ++(2.0pt,-2.0pt);
\draw [color=ffxfqq] (401.,0.6205192805)-- ++(-1.0pt,-1.0pt) -- ++(2.0pt,2.0pt) ++(-2.0pt,0) -- ++(2.0pt,-2.0pt);
\draw [color=ffxfqq] (431.,0.6488430356)-- ++(-1.0pt,-1.0pt) -- ++(2.0pt,2.0pt) ++(-2.0pt,0) -- ++(2.0pt,-2.0pt);
\draw [color=ffxfqq] (449.,0.9376236037)-- ++(-1.0pt,-1.0pt) -- ++(2.0pt,2.0pt) ++(-2.0pt,0) -- ++(2.0pt,-2.0pt);
\draw [color=ffxfqq] (461.,0.3658510685)-- ++(-1.0pt,-1.0pt) -- ++(2.0pt,2.0pt) ++(-2.0pt,0) -- ++(2.0pt,-2.0pt);
\draw [color=ffxfqq] (467.,-1.4909625509)-- ++(-1.0pt,-1.0pt) -- ++(2.0pt,2.0pt) ++(-2.0pt,0) -- ++(2.0pt,-2.0pt);
\draw [color=ffxfqq] (479.,0.0298720866)-- ++(-1.0pt,-1.0pt) -- ++(2.0pt,2.0pt) ++(-2.0pt,0) -- ++(2.0pt,-2.0pt);
\draw [color=ffxfqq] (491.,0.0302367328)-- ++(-1.0pt,-1.0pt) -- ++(2.0pt,2.0pt) ++(-2.0pt,0) -- ++(2.0pt,-2.0pt);
\draw [color=ffxfqq] (503.,0.8018548879)-- ++(-1.0pt,-1.0pt) -- ++(2.0pt,2.0pt) ++(-2.0pt,0) -- ++(2.0pt,-2.0pt);
\draw [color=ffxfqq] (509.,-1.0318126321)-- ++(-1.0pt,-1.0pt) -- ++(2.0pt,2.0pt) ++(-2.0pt,0) -- ++(2.0pt,-2.0pt);
\draw [color=ffxfqq] (521.,1.4937006612)-- ++(-1.0pt,-1.0pt) -- ++(2.0pt,2.0pt) ++(-2.0pt,0) -- ++(2.0pt,-2.0pt);
\draw [color=ffxfqq] (557.,0.4078019416)-- ++(-1.0pt,-1.0pt) -- ++(2.0pt,2.0pt) ++(-2.0pt,0) -- ++(2.0pt,-2.0pt);
\draw [color=ffxfqq] (563.,1.2717689145)-- ++(-1.0pt,-1.0pt) -- ++(2.0pt,2.0pt) ++(-2.0pt,0) -- ++(2.0pt,-2.0pt);
\draw [color=ffxfqq] (569.,0.6719744007)-- ++(-1.0pt,-1.0pt) -- ++(2.0pt,2.0pt) ++(-2.0pt,0) -- ++(2.0pt,-2.0pt);
\draw [color=ffxfqq] (587.,-1.0693016251)-- ++(-1.0pt,-1.0pt) -- ++(2.0pt,2.0pt) ++(-2.0pt,0) -- ++(2.0pt,-2.0pt);
\draw [color=ffxfqq] (593.,-0.2637154256)-- ++(-1.0pt,-1.0pt) -- ++(2.0pt,2.0pt) ++(-2.0pt,0) -- ++(2.0pt,-2.0pt);
\draw [color=ffxfqq] (599.,0.635059895)-- ++(-1.0pt,-1.0pt) -- ++(2.0pt,2.0pt) ++(-2.0pt,0) -- ++(2.0pt,-2.0pt);
\draw [color=ffxfqq] (617.,-0.1966334537)-- ++(-1.0pt,-1.0pt) -- ++(2.0pt,2.0pt) ++(-2.0pt,0) -- ++(2.0pt,-2.0pt);
\draw [color=ffxfqq] (641.,0.947791976)-- ++(-1.0pt,-1.0pt) -- ++(2.0pt,2.0pt) ++(-2.0pt,0) -- ++(2.0pt,-2.0pt);
\draw [color=ffxfqq] (647.,0.0745244618)-- ++(-1.0pt,-1.0pt) -- ++(2.0pt,2.0pt) ++(-2.0pt,0) -- ++(2.0pt,-2.0pt);
\draw [color=ffxfqq] (653.,-0.3562215691)-- ++(-1.0pt,-1.0pt) -- ++(2.0pt,2.0pt) ++(-2.0pt,0) -- ++(2.0pt,-2.0pt);
\draw [color=ffxfqq] (659.,-0.5249702239)-- ++(-1.0pt,-1.0pt) -- ++(2.0pt,2.0pt) ++(-2.0pt,0) -- ++(2.0pt,-2.0pt);
\draw [color=ffxfqq] (677.,-0.844991943)-- ++(-1.0pt,-1.0pt) -- ++(2.0pt,2.0pt) ++(-2.0pt,0) -- ++(2.0pt,-2.0pt);
\draw [color=ffxfqq] (683.,0.4693221606)-- ++(-1.0pt,-1.0pt) -- ++(2.0pt,2.0pt) ++(-2.0pt,0) -- ++(2.0pt,-2.0pt);
\draw [color=ffxfqq] (719.,-0.0689565863)-- ++(-1.0pt,-1.0pt) -- ++(2.0pt,2.0pt) ++(-2.0pt,0) -- ++(2.0pt,-2.0pt);
\draw [color=ffxfqq] (743.,-1.0909618098)-- ++(-1.0pt,-1.0pt) -- ++(2.0pt,2.0pt) ++(-2.0pt,0) -- ++(2.0pt,-2.0pt);
\draw [color=ffxfqq] (761.,-0.8228927805)-- ++(-1.0pt,-1.0pt) -- ++(2.0pt,2.0pt) ++(-2.0pt,0) -- ++(2.0pt,-2.0pt);
\draw [color=ffxfqq] (773.,0.5874825371)-- ++(-1.0pt,-1.0pt) -- ++(2.0pt,2.0pt) ++(-2.0pt,0) -- ++(2.0pt,-2.0pt);
\draw [color=ffxfqq] (797.,-0.3517407238)-- ++(-1.0pt,-1.0pt) -- ++(2.0pt,2.0pt) ++(-2.0pt,0) -- ++(2.0pt,-2.0pt);
\draw [color=ffxfqq] (809.,1.303900962)-- ++(-1.0pt,-1.0pt) -- ++(2.0pt,2.0pt) ++(-2.0pt,0) -- ++(2.0pt,-2.0pt);
\draw [color=ffxfqq] (821.,-0.0255805241)-- ++(-1.0pt,-1.0pt) -- ++(2.0pt,2.0pt) ++(-2.0pt,0) -- ++(2.0pt,-2.0pt);
\draw [color=ffxfqq] (827.,0.143047913)-- ++(-1.0pt,-1.0pt) -- ++(2.0pt,2.0pt) ++(-2.0pt,0) -- ++(2.0pt,-2.0pt);
\draw [color=ffxfqq] (839.,-0.1402061381)-- ++(-1.0pt,-1.0pt) -- ++(2.0pt,2.0pt) ++(-2.0pt,0) -- ++(2.0pt,-2.0pt);
\draw [color=ffxfqq] (863.,3.5229985309)-- ++(-1.0pt,-1.0pt) -- ++(2.0pt,2.0pt) ++(-2.0pt,0) -- ++(2.0pt,-2.0pt);
\draw [color=ffxfqq] (881.,-0.7333143075)-- ++(-1.0pt,-1.0pt) -- ++(2.0pt,2.0pt) ++(-2.0pt,0) -- ++(2.0pt,-2.0pt);
\draw [color=ffxfqq] (941.,1.0591219161)-- ++(-1.0pt,-1.0pt) -- ++(2.0pt,2.0pt) ++(-2.0pt,0) -- ++(2.0pt,-2.0pt);
\draw [color=ffxfqq] (947.,0.1401080697)-- ++(-1.0pt,-1.0pt) -- ++(2.0pt,2.0pt) ++(-2.0pt,0) -- ++(2.0pt,-2.0pt);
\draw [color=ffxfqq] (953.,1.7567872826)-- ++(-1.0pt,-1.0pt) -- ++(2.0pt,2.0pt) ++(-2.0pt,0) -- ++(2.0pt,-2.0pt);
\draw [color=ffxfqq] (983.,1.1314422806)-- ++(-1.0pt,-1.0pt) -- ++(2.0pt,2.0pt) ++(-2.0pt,0) -- ++(2.0pt,-2.0pt);
\draw [color=ffxfqq] (1013.,-2.8012041185)-- ++(-1.0pt,-1.0pt) -- ++(2.0pt,2.0pt) ++(-2.0pt,0) -- ++(2.0pt,-2.0pt);
\draw [color=ffxfqq] (1019.,-1.4158226895)-- ++(-1.0pt,-1.0pt) -- ++(2.0pt,2.0pt) ++(-2.0pt,0) -- ++(2.0pt,-2.0pt);
\draw [color=ffxfqq] (1031.,0.859875403)-- ++(-1.0pt,-1.0pt) -- ++(2.0pt,2.0pt) ++(-2.0pt,0) -- ++(2.0pt,-2.0pt);
\draw [color=ffxfqq] (1049.,-1.363137533)-- ++(-1.0pt,-1.0pt) -- ++(2.0pt,2.0pt) ++(-2.0pt,0) -- ++(2.0pt,-2.0pt);
\draw [color=ffxfqq] (1061.,0.8003739711)-- ++(-1.0pt,-1.0pt) -- ++(2.0pt,2.0pt) ++(-2.0pt,0) -- ++(2.0pt,-2.0pt);
\draw [color=ffxfqq] (1097.,0.6374445004)-- ++(-1.0pt,-1.0pt) -- ++(2.0pt,2.0pt) ++(-2.0pt,0) -- ++(2.0pt,-2.0pt);
\draw [color=ffxfqq] (1103.,1.2339034821)-- ++(-1.0pt,-1.0pt) -- ++(2.0pt,2.0pt) ++(-2.0pt,0) -- ++(2.0pt,-2.0pt);
\draw [color=ffxfqq] (1151.,1.7882627568)-- ++(-1.0pt,-1.0pt) -- ++(2.0pt,2.0pt) ++(-2.0pt,0) -- ++(2.0pt,-2.0pt);
\draw [color=ffxfqq] (1163.,-1.4342327159)-- ++(-1.0pt,-1.0pt) -- ++(2.0pt,2.0pt) ++(-2.0pt,0) -- ++(2.0pt,-2.0pt);
\draw [color=ffxfqq] (1181.,-1.0908656857)-- ++(-1.0pt,-1.0pt) -- ++(2.0pt,2.0pt) ++(-2.0pt,0) -- ++(2.0pt,-2.0pt);
\draw [color=ffxfqq] (1187.,-1.6004810259)-- ++(-1.0pt,-1.0pt) -- ++(2.0pt,2.0pt) ++(-2.0pt,0) -- ++(2.0pt,-2.0pt);
\draw [color=ffxfqq] (1229.,3.1518826273)-- ++(-1.0pt,-1.0pt) -- ++(2.0pt,2.0pt) ++(-2.0pt,0) -- ++(2.0pt,-2.0pt);
\draw [color=ffxfqq] (1277.,0.4673675216)-- ++(-1.0pt,-1.0pt) -- ++(2.0pt,2.0pt) ++(-2.0pt,0) -- ++(2.0pt,-2.0pt);
\draw [color=ffxfqq] (1301.,0.056961017)-- ++(-1.0pt,-1.0pt) -- ++(2.0pt,2.0pt) ++(-2.0pt,0) -- ++(2.0pt,-2.0pt);
\draw [color=ffxfqq] (1307.,0.345366543)-- ++(-1.0pt,-1.0pt) -- ++(2.0pt,2.0pt) ++(-2.0pt,0) -- ++(2.0pt,-2.0pt);
\draw [color=ffxfqq] (1319.,2.6144288154)-- ++(-1.0pt,-1.0pt) -- ++(2.0pt,2.0pt) ++(-2.0pt,0) -- ++(2.0pt,-2.0pt);
\draw [color=ffxfqq] (1361.,0.5260202429)-- ++(-1.0pt,-1.0pt) -- ++(2.0pt,2.0pt) ++(-2.0pt,0) -- ++(2.0pt,-2.0pt);
\draw [color=ffxfqq] (1367.,1.7001276776)-- ++(-1.0pt,-1.0pt) -- ++(2.0pt,2.0pt) ++(-2.0pt,0) -- ++(2.0pt,-2.0pt);
\draw [color=ffxfqq] (1373.,1.779841835)-- ++(-1.0pt,-1.0pt) -- ++(2.0pt,2.0pt) ++(-2.0pt,0) -- ++(2.0pt,-2.0pt);
\draw [color=ffxfqq] (1409.,-0.874488022)-- ++(-1.0pt,-1.0pt) -- ++(2.0pt,2.0pt) ++(-2.0pt,0) -- ++(2.0pt,-2.0pt);
\draw [color=ffxfqq] (1427.,-2.6618872033)-- ++(-1.0pt,-1.0pt) -- ++(2.0pt,2.0pt) ++(-2.0pt,0) -- ++(2.0pt,-2.0pt);
\draw [color=ffxfqq] (1433.,-1.1325830695)-- ++(-1.0pt,-1.0pt) -- ++(2.0pt,2.0pt) ++(-2.0pt,0) -- ++(2.0pt,-2.0pt);
\draw [color=ffxfqq] (1451.,0.1891777264)-- ++(-1.0pt,-1.0pt) -- ++(2.0pt,2.0pt) ++(-2.0pt,0) -- ++(2.0pt,-2.0pt);
\draw [color=ffxfqq] (1481.,1.0356002519)-- ++(-1.0pt,-1.0pt) -- ++(2.0pt,2.0pt) ++(-2.0pt,0) -- ++(2.0pt,-2.0pt);
\draw [color=ffxfqq] (1499.,1.3641814971)-- ++(-1.0pt,-1.0pt) -- ++(2.0pt,2.0pt) ++(-2.0pt,0) -- ++(2.0pt,-2.0pt);
\draw [color=ffxfqq] (1511.,2.2787121485)-- ++(-1.0pt,-1.0pt) -- ++(2.0pt,2.0pt) ++(-2.0pt,0) -- ++(2.0pt,-2.0pt);
\draw [color=ffxfqq] (1523.,0.1164210948)-- ++(-1.0pt,-1.0pt) -- ++(2.0pt,2.0pt) ++(-2.0pt,0) -- ++(2.0pt,-2.0pt);
\draw [color=ffxfqq] (1553.,1.6701710231)-- ++(-1.0pt,-1.0pt) -- ++(2.0pt,2.0pt) ++(-2.0pt,0) -- ++(2.0pt,-2.0pt);
\draw [color=ffxfqq] (1559.,-0.1613015958)-- ++(-1.0pt,-1.0pt) -- ++(2.0pt,2.0pt) ++(-2.0pt,0) -- ++(2.0pt,-2.0pt);
\draw [color=ffxfqq] (1571.,0.1486334805)-- ++(-1.0pt,-1.0pt) -- ++(2.0pt,2.0pt) ++(-2.0pt,0) -- ++(2.0pt,-2.0pt);
\draw [color=ffxfqq] (1583.,1.7659032609)-- ++(-1.0pt,-1.0pt) -- ++(2.0pt,2.0pt) ++(-2.0pt,0) -- ++(2.0pt,-2.0pt);
\draw [color=ffxfqq] (1601.,-1.9178874104)-- ++(-1.0pt,-1.0pt) -- ++(2.0pt,2.0pt) ++(-2.0pt,0) -- ++(2.0pt,-2.0pt);
\draw [color=ffxfqq] (1607.,-0.0453540593)-- ++(-1.0pt,-1.0pt) -- ++(2.0pt,2.0pt) ++(-2.0pt,0) -- ++(2.0pt,-2.0pt);
\draw [color=ffxfqq] (1619.,0.4384716041)-- ++(-1.0pt,-1.0pt) -- ++(2.0pt,2.0pt) ++(-2.0pt,0) -- ++(2.0pt,-2.0pt);
\draw [color=ffxfqq] (1667.,1.0194698955)-- ++(-1.0pt,-1.0pt) -- ++(2.0pt,2.0pt) ++(-2.0pt,0) -- ++(2.0pt,-2.0pt);
\draw [color=ffxfqq] (1697.,-0.0575126998)-- ++(-1.0pt,-1.0pt) -- ++(2.0pt,2.0pt) ++(-2.0pt,0) -- ++(2.0pt,-2.0pt);
\draw [color=ffxfqq] (1709.,0.8219953419)-- ++(-1.0pt,-1.0pt) -- ++(2.0pt,2.0pt) ++(-2.0pt,0) -- ++(2.0pt,-2.0pt);
\draw [color=ffxfqq] (1721.,-0.8744169616)-- ++(-1.0pt,-1.0pt) -- ++(2.0pt,2.0pt) ++(-2.0pt,0) -- ++(2.0pt,-2.0pt);
\draw [color=ffxfqq] (1733.,-1.4319584204)-- ++(-1.0pt,-1.0pt) -- ++(2.0pt,2.0pt) ++(-2.0pt,0) -- ++(2.0pt,-2.0pt);
\draw [color=ffxfqq] (1787.,1.1311319532)-- ++(-1.0pt,-1.0pt) -- ++(2.0pt,2.0pt) ++(-2.0pt,0) -- ++(2.0pt,-2.0pt);
\draw [color=ffxfqq] (1823.,-2.9905613373)-- ++(-1.0pt,-1.0pt) -- ++(2.0pt,2.0pt) ++(-2.0pt,0) -- ++(2.0pt,-2.0pt);
\draw [color=ffxfqq] (1871.,0.7329387643)-- ++(-1.0pt,-1.0pt) -- ++(2.0pt,2.0pt) ++(-2.0pt,0) -- ++(2.0pt,-2.0pt);
\draw [color=ffxfqq] (1877.,0.2279763372)-- ++(-1.0pt,-1.0pt) -- ++(2.0pt,2.0pt) ++(-2.0pt,0) -- ++(2.0pt,-2.0pt);
\draw [color=ffxfqq] (1889.,0.570428723)-- ++(-1.0pt,-1.0pt) -- ++(2.0pt,2.0pt) ++(-2.0pt,0) -- ++(2.0pt,-2.0pt);
\draw [color=ffxfqq] (1907.,0.996925624)-- ++(-1.0pt,-1.0pt) -- ++(2.0pt,2.0pt) ++(-2.0pt,0) -- ++(2.0pt,-2.0pt);
\draw [color=ffxfqq] (1913.,-3.6081316047)-- ++(-1.0pt,-1.0pt) -- ++(2.0pt,2.0pt) ++(-2.0pt,0) -- ++(2.0pt,-2.0pt);
\draw [color=ffxfqq] (1931.,0.411232203)-- ++(-1.0pt,-1.0pt) -- ++(2.0pt,2.0pt) ++(-2.0pt,0) -- ++(2.0pt,-2.0pt);
\draw [color=ffxfqq] (1949.,0.0948796333)-- ++(-1.0pt,-1.0pt) -- ++(2.0pt,2.0pt) ++(-2.0pt,0) -- ++(2.0pt,-2.0pt);
\draw [color=ffxfqq] (1973.,0.6231716048)-- ++(-1.0pt,-1.0pt) -- ++(2.0pt,2.0pt) ++(-2.0pt,0) -- ++(2.0pt,-2.0pt);
\draw [color=ffxfqq] (1979.,0.7025695646)-- ++(-1.0pt,-1.0pt) -- ++(2.0pt,2.0pt) ++(-2.0pt,0) -- ++(2.0pt,-2.0pt);
\draw [color=ffxfqq] (2027.,0.343793495)-- ++(-1.0pt,-1.0pt) -- ++(2.0pt,2.0pt) ++(-2.0pt,0) -- ++(2.0pt,-2.0pt);
\draw [color=ffxfqq] (2039.,0.3800214189)-- ++(-1.0pt,-1.0pt) -- ++(2.0pt,2.0pt) ++(-2.0pt,0) -- ++(2.0pt,-2.0pt);
\draw [color=ffxfqq] (2063.,-0.9972370202)-- ++(-1.0pt,-1.0pt) -- ++(2.0pt,2.0pt) ++(-2.0pt,0) -- ++(2.0pt,-2.0pt);
\draw [color=ffxfqq] (2069.,-1.2770046326)-- ++(-1.0pt,-1.0pt) -- ++(2.0pt,2.0pt) ++(-2.0pt,0) -- ++(2.0pt,-2.0pt);
\draw [color=ffxfqq] (2081.,2.3785171312)-- ++(-1.0pt,-1.0pt) -- ++(2.0pt,2.0pt) ++(-2.0pt,0) -- ++(2.0pt,-2.0pt);
\draw [color=ffxfqq] (2099.,0.6189014531)-- ++(-1.0pt,-1.0pt) -- ++(2.0pt,2.0pt) ++(-2.0pt,0) -- ++(2.0pt,-2.0pt);
\draw [color=ffxfqq] (2111.,-2.0323239781)-- ++(-1.0pt,-1.0pt) -- ++(2.0pt,2.0pt) ++(-2.0pt,0) -- ++(2.0pt,-2.0pt);
\draw [color=ffxfqq] (2129.,0.3740976413)-- ++(-1.0pt,-1.0pt) -- ++(2.0pt,2.0pt) ++(-2.0pt,0) -- ++(2.0pt,-2.0pt);
\draw [color=ffxfqq] (2141.,-0.8930326914)-- ++(-1.0pt,-1.0pt) -- ++(2.0pt,2.0pt) ++(-2.0pt,0) -- ++(2.0pt,-2.0pt);
\draw [color=ffxfqq] (2153.,0.5673615494)-- ++(-1.0pt,-1.0pt) -- ++(2.0pt,2.0pt) ++(-2.0pt,0) -- ++(2.0pt,-2.0pt);
\draw [color=ffxfqq] (2207.,0.0224094154)-- ++(-1.0pt,-1.0pt) -- ++(2.0pt,2.0pt) ++(-2.0pt,0) -- ++(2.0pt,-2.0pt);
\draw [color=ffxfqq] (2213.,-0.134961301)-- ++(-1.0pt,-1.0pt) -- ++(2.0pt,2.0pt) ++(-2.0pt,0) -- ++(2.0pt,-2.0pt);
\draw [color=ffxfqq] (2237.,-0.3050084562)-- ++(-1.0pt,-1.0pt) -- ++(2.0pt,2.0pt) ++(-2.0pt,0) -- ++(2.0pt,-2.0pt);
\draw [color=ffxfqq] (2267.,-0.2813773355)-- ++(-1.0pt,-1.0pt) -- ++(2.0pt,2.0pt) ++(-2.0pt,0) -- ++(2.0pt,-2.0pt);
\draw [color=ffxfqq] (2273.,1.4698037127)-- ++(-1.0pt,-1.0pt) -- ++(2.0pt,2.0pt) ++(-2.0pt,0) -- ++(2.0pt,-2.0pt);
\draw [color=ffxfqq] (2297.,-0.5264596992)-- ++(-1.0pt,-1.0pt) -- ++(2.0pt,2.0pt) ++(-2.0pt,0) -- ++(2.0pt,-2.0pt);
\draw [color=ffxfqq] (2309.,2.5511369013)-- ++(-1.0pt,-1.0pt) -- ++(2.0pt,2.0pt) ++(-2.0pt,0) -- ++(2.0pt,-2.0pt);
\draw [color=ffxfqq] (2333.,-1.0020878652)-- ++(-1.0pt,-1.0pt) -- ++(2.0pt,2.0pt) ++(-2.0pt,0) -- ++(2.0pt,-2.0pt);
\draw [color=ffxfqq] (2339.,1.0749191048)-- ++(-1.0pt,-1.0pt) -- ++(2.0pt,2.0pt) ++(-2.0pt,0) -- ++(2.0pt,-2.0pt);
\draw [color=ffxfqq] (2351.,3.0424973615)-- ++(-1.0pt,-1.0pt) -- ++(2.0pt,2.0pt) ++(-2.0pt,0) -- ++(2.0pt,-2.0pt);
\draw [color=ffxfqq] (2357.,0.9382605141)-- ++(-1.0pt,-1.0pt) -- ++(2.0pt,2.0pt) ++(-2.0pt,0) -- ++(2.0pt,-2.0pt);
\draw [color=ffxfqq] (2381.,-0.4096479762)-- ++(-1.0pt,-1.0pt) -- ++(2.0pt,2.0pt) ++(-2.0pt,0) -- ++(2.0pt,-2.0pt);
\draw [color=ffxfqq] (2393.,-0.8628808376)-- ++(-1.0pt,-1.0pt) -- ++(2.0pt,2.0pt) ++(-2.0pt,0) -- ++(2.0pt,-2.0pt);
\draw [color=ffxfqq] (2399.,0.6357257727)-- ++(-1.0pt,-1.0pt) -- ++(2.0pt,2.0pt) ++(-2.0pt,0) -- ++(2.0pt,-2.0pt);
\draw [color=ffxfqq] (2411.,0.0773654226)-- ++(-1.0pt,-1.0pt) -- ++(2.0pt,2.0pt) ++(-2.0pt,0) -- ++(2.0pt,-2.0pt);
\draw [color=ffxfqq] (2417.,-0.6264087852)-- ++(-1.0pt,-1.0pt) -- ++(2.0pt,2.0pt) ++(-2.0pt,0) -- ++(2.0pt,-2.0pt);
\draw [color=ffxfqq] (2441.,-0.4606264266)-- ++(-1.0pt,-1.0pt) -- ++(2.0pt,2.0pt) ++(-2.0pt,0) -- ++(2.0pt,-2.0pt);
\draw [color=ffxfqq] (2447.,-0.5964489482)-- ++(-1.0pt,-1.0pt) -- ++(2.0pt,2.0pt) ++(-2.0pt,0) -- ++(2.0pt,-2.0pt);
\draw [color=ffxfqq] (2459.,0.9986614253)-- ++(-1.0pt,-1.0pt) -- ++(2.0pt,2.0pt) ++(-2.0pt,0) -- ++(2.0pt,-2.0pt);
\draw [color=ffxfqq] (2531.,0.4725236784)-- ++(-1.0pt,-1.0pt) -- ++(2.0pt,2.0pt) ++(-2.0pt,0) -- ++(2.0pt,-2.0pt);
\end{scriptsize}
\end{tikzpicture}

\caption{Values of both $\frac{C_d-M^*_{d}}{d}$ (on the left) and $\frac{C_d-M^{**}_{d}}{d}$ (on the right) for the selected prime numbers $d$.}
\label{Fig1}
\end{figure}

\subsubsection{A first test in each case}\label{ss2}

From the $152$ left plotted values on Figure \ref{Fig1}, we draw the regression line: its slope $s^*$ is approximately $1.94\times10^{-4}$ and its 
intercept is approximately $0.352$. Let us consider the following null hypothesis $H^*_0$: $s^*=0$. We have to calculate $T^*=\frac{s^*-0}
{\hat\sigma_{s^*}}$, where $\hat\sigma_{s^*}$ is the estimated standard deviation of the slope. We obtain $\hat\sigma_{s^*}\approx 1.35\times10^{-4}$ 
and $T^*\approx 1.43$. $T^*$ follows a student's t-distribution with $(152-2)$ degrees of freedom \cite[Proposition 1.8]{cor}. The acceptance region of 
the hypothesis test with a $5\%$ risk is approximately $[-1.976,1.976]$. Thus it can be concluded that we cannot reject the null hypothesis, 
\textit{i.e.} the fact that the slope $s^*$ is not significantly different from zero.

From the $153$ right plotted values on Figure \ref{Fig1}, we draw the regression line: its slope $s^{**}$ is approximately $-1.88\times10^{-4}$ and its 
intercept is approximately $0.511$. Let us consider the following null hypothesis $H^{**}_0$: $s^{**}=0$. We again have to calculate $T^{**}=
\frac{s^{**}-0}{\hat\sigma_{s^{**}}}$. We here obtain $\hat\sigma_{s^{**}}\approx 1.25\times10^{-4}$ and $T^{**}\approx -1.50$. $T^{**}$ follows a 
student's t-distribution with $(153-2)$ degrees of freedom. The acceptance region of the hypothesis test with a $5\%$ risk is approximately 
$[-1.976,1.976]$. Thus it can be concluded that we cannot reject the fact that the slope $s^{**}$ is not significantly different from zero.

\subsubsection{A set of tests in each case}

Figure \ref{Fig3} below shows the distribution of the values of $\frac{C_d-M^*_d}{d}$ (on the left) and $\frac{C_d-M^{**}_d}{d}$ (on the right) for the 
considered values of $d$.

\begin{figure}[ht]
\centering
\definecolor{xfqqff}{rgb}{0.5,0.,1.}
\begin{tikzpicture}[line cap=round,line join=round,>=triangle 45,x=0.6cm,y=0.3cm]
\draw[->,color=black] (-3.535852581203811,0.) -- (5,0.);
\foreach \x in {-3,-2,-1,1,2,3,4}\draw[shift={(\x,0)},color=black] (0pt,2pt) -- (0pt,-2pt) node[below] {\footnotesize $\x$};
\draw[->,color=black] (0.,-0.75) -- (0.,19.1);
\foreach \y in {2,4,6,8,10,12,14,16,18}\draw[shift={(0,\y)},color=black] (2pt,0pt) -- (-2pt,0pt) node[left] {\footnotesize $\y$};
\draw[color=black] (-3pt,-8.3pt) node {\footnotesize $0$};\clip(-3.535852581203811,-0.75) rectangle (5,19.1);
\fill[color=xfqqff,fill=xfqqff,fill opacity=0.5] (-3.,0.) -- (-3.,2.) -- (-2.75,2.) -- (-2.75,0.) -- cycle;
\fill[color=xfqqff,fill=xfqqff,fill opacity=0.5] (4.25,0.) -- (4.25,1.) -- (4.5,1.) -- (4.5,0.) -- cycle;
\fill[color=xfqqff,fill=xfqqff,fill opacity=0.5] (4.,1.) -- (3.,1.) -- (3.,0.) -- (4.,0.) -- cycle;
\fill[color=xfqqff,fill=xfqqff,fill opacity=0.5] (2.75,0.) -- (-2.25,0.) -- (-2.25,2.) -- (-2.,2.) -- (-2.,1.) -- (-1.75,1.) -- (-1.75,3.) -- (-1.5,3.) -- (-1.5,2.) -- (-1.25,2.) -- (-1.25,4.) -- (-1.,4.) -- (-1.,7.) -- (-0.5,7.) -- (-0.5,6.) -- (-0.25,6.) -- (-0.25,12.) -- (0.,12.) -- (0.,15.) -- (0.25,15.) -- (0.25,8.) -- (0.5,8.) -- (0.5,14.) -- (1.,14.) -- (1.,11.) -- (1.25,11.) -- (1.25,12.) -- (1.5,12.) -- (1.5,9.) -- (1.75,9.) -- (1.75,8.) -- (2.,8.) -- (2.,4.) -- (2.5,4.) -- (2.5,2.) -- (2.75,2.) -- cycle;
\draw (-3.,2.)-- (-2.75,2.);\draw (-2.75,2.)-- (-2.75,0.);\draw (-2.25,0.)-- (-2.25,2.);\draw (-2.25,2.)-- (-2.,2.);\draw (-3.,2.)-- (-3.,0.);
\draw (-2.,1.)-- (-1.75,1.);\draw (3.,1.)-- (3.,0.);\draw (-2.,2.)-- (-2.,0.);\draw (-1.75,3.)-- (-1.5,3.);\draw (-1.5,2.)-- (-1.25,2.);
\draw (-1.25,4.)-- (-1.,4.);\draw (-1.75,3.)-- (-1.75,0.);\draw (-1.5,0.)-- (-1.5,3.);\draw (-1.25,0.)-- (-1.25,2.);\draw (-1.25,4.)-- (-1.25,2.);
\draw (-1.,7.)-- (-1.,0.);\draw (-1.,7.)-- (-0.5,7.);\draw (2.5,2.)-- (2.75,2.);\draw (2.5,4.)-- (2.5,0.);\draw (2.5,4.)-- (2.,4.);
\draw (2.25,4.)-- (2.25,0.);\draw (4.25,1.)-- (4.5,1.);\draw (4.5,0.)-- (4.5,1.);\draw (4.25,0.)-- (4.25,1.);\draw (4.,0.)-- (4.,1.);
\draw (3.75,0.)-- (3.75,1.);\draw (3.5,0.)-- (3.5,1.);\draw (3.25,0.)-- (3.25,1.);\draw (3.,0.)-- (3.,1.);\draw (2.75,2.)-- (2.75,0.);
\draw (3.,1.)-- (4.,1.);\draw (0.25,8.)-- (0.5,8.);\draw (-0.5,6.)-- (-0.25,6.);\draw (-0.25,12.)-- (0.,12.);\draw (1.25,12.)-- (1.5,12.);
\draw (1.5,9.)-- (1.75,9.);\draw (1.75,8.)-- (2.,8.);\draw (1.,11.)-- (1.25,11.);\draw (1.25,12.)-- (1.25,0.);\draw (1.5,12.)-- (1.5,0.);
\draw (1.75,9.)-- (1.75,0.);\draw (2.,8.)-- (2.,0.);\draw (-0.75,7.)-- (-0.75,0.);\draw (-0.5,7.)-- (-0.5,0.);\draw (-0.25,12.)-- (-0.25,0.);
\draw (0.5,14.)-- (1.,14.);\draw (1.,14.)-- (1.,0.);\draw (0.75,14.)-- (0.75,0.);\draw (0.5,14.)-- (0.5,0.);\draw (0.,15.)-- (0.,0.);
\draw (0.25,15.)-- (0.25,0.);\draw (0.,15.)-- (0.25,15.);\draw (3.65,4.3) node {$\frac{C_d-M^*_d}{d}$};
\draw (-2.2,17) node {$Frequency$};
\end{tikzpicture}
\definecolor{ffxfqq}{rgb}{1.,0.5,0.}
\begin{tikzpicture}[line cap=round,line join=round,>=triangle 45,x=0.6cm,y=0.3cm]
\draw[->,color=black] (-4.16,0.) -- (4.4,0.);
\foreach \x in {-4,-3,-2,-1,1,2,3}\draw[shift={(\x,0)},color=black] (0pt,2pt) -- (0pt,-2pt) node[below] {\footnotesize $\x$};
\draw[->,color=black] (0.,-0.75) -- (0.,19.1);
\foreach \y in {2,4,6,8,10,12,14,16,18}\draw[shift={(0,\y)},color=black] (2pt,0pt) -- (-2pt,0pt) node[left] {\footnotesize $\y$};
\draw[color=black] (-3pt,-8.3pt) node {\footnotesize $0$};\clip(-4.16,-0.75) rectangle (4.2,19.1);
\fill[color=ffxfqq,fill=ffxfqq,fill opacity=0.5] (-3.75,0.) -- (-3.75,1.) -- (-3.5,1.) -- (-3.5,0.) -- cycle;
\fill[color=ffxfqq,fill=ffxfqq,fill opacity=0.5] (-3.,0.) -- (-3.,2.) -- (-2.5,2.) -- (-2.5,0.) -- cycle;
\fill[color=ffxfqq,fill=ffxfqq,fill opacity=0.5] (3.5,1.) -- (3.75,1.) -- (3.75,0.) -- (3.5,0.) -- cycle;
\fill[color=ffxfqq,fill=ffxfqq,fill opacity=0.5] (3.25,2.) -- (3.25,0.) -- (3.,0.) -- (3.,2.) -- cycle;
\fill[color=ffxfqq,fill=ffxfqq,fill opacity=0.5] (2.5,3.) -- (2.75,3.) -- (2.75,0.) -- (2.25,0.) -- (2.25,2.) -- (2.5,2.) -- cycle;
\fill[color=ffxfqq,fill=ffxfqq,fill opacity=0.5] (2.,0.) -- (-2.25,0.) -- (-2.25,1.) -- (-1.75,1.) -- (-1.75,2.) -- (-1.5,2.) -- (-1.5,6.) -- (-1.25,6.) -- (-1.25,7.) -- (-0.25,7.) -- (-0.25,9.) -- (0.,9.) -- (0.,17.) -- (0.25,17.) -- (0.25,13.) -- (0.5,13.) -- (0.5,19.) -- (0.75,19.) -- (0.75,13.) -- (1.,13.) -- (1.,10.) -- (1.25,10.) -- (1.25,8.) -- (1.5,8.) -- (1.5,9.) -- (1.75,9.) -- (1.75,4.) -- (2.,4.) -- cycle;
\draw (3.2,4.3) node {$\frac{C_d-M^{**}_d}{d}$};\draw (-2.4,17) node {$Frequency$};
\draw (-3.75,1.)-- (-3.5,1.);\draw (-3.,2.)-- (-2.5,2.);\draw (-2.25,1.)-- (-1.75,1.);\draw (-1.75,2.)-- (-1.5,2.);\draw (-1.5,6.)-- (-1.25,6.);
\draw (-1.25,7.)-- (-0.25,7.);\draw (-0.25,9.)-- (0.,9.);\draw (0.5,19.)-- (0.75,19.);\draw (0.75,13.)-- (1.,13.);\draw (0.25,13.)-- (0.5,13.);
\draw (0.,17.)-- (0.25,17.);\draw (1.,10.)-- (1.25,10.);\draw (1.25,8.)-- (1.5,8.);\draw (1.75,9.)-- (1.5,9.);\draw (1.75,4.)-- (2.,4.);
\draw (2.25,2.)-- (2.5,2.);\draw (2.75,3.)-- (2.5,3.);\draw (3.,2.)-- (3.25,2.);\draw (3.5,1.)-- (3.75,1.);\draw (3.,2.)-- (3.,0.);
\draw (2.75,0.)-- (2.75,3.);\draw (3.25,2.)-- (3.25,0.);\draw (3.5,0.)-- (3.5,1.);\draw (2.25,2.)-- (2.25,0.);\draw (2.,0.)-- (2.,4.);
\draw (-3.5,1.)-- (-3.5,0.);\draw (-3.,2.)-- (-3.,0.);\draw (-2.5,2.)-- (-2.5,0.);\draw (-2.25,1.)-- (-2.25,0.);\draw (-1.5,2.)-- (-1.5,0.);
\draw (-1.5,6.)-- (-1.5,2.);\draw (-1.75,2.)-- (-1.75,0.);\draw (-2.,1.)-- (-2.,0.);\draw (-1.25,7.)-- (-1.25,0.);\draw (-2.75,2.)-- (-2.75,0.);
\draw (-3.75,1.)-- (-3.75,0.);\draw (-1.,7.)-- (-1.,0.);\draw (-0.75,0.)-- (-0.75,7.);\draw (-0.5,7.)-- (-0.5,0.);\draw (-0.25,0.)-- (-0.25,9.);
\draw (0.25,17.)-- (0.25,0.);\draw (0.,17.)-- (0.,0.);\draw (0.5,0.)-- (0.5,13.);\draw (0.5,19.)-- (0.5,13.);\draw (0.75,19.)-- (0.75,0.);
\draw (1.,0.)-- (1.,13.);\draw (1.25,10.)-- (1.25,0.);\draw (1.5,0.)-- (1.5,9.);\draw (1.75,9.)-- (1.75,0.);\draw (2.5,3.)-- (2.5,0.);
\draw (3.75,0.)-- (3.75,1.);
\end{tikzpicture}
\caption{Distribution of values of both $\frac{C_d-M^*_{d}}{d}$ (on the left) and $\frac{C_d-M^{**}_{d}}{d}$ (on the right) for the selected prime 
numbers $d$.}
\label{Fig3}
\end{figure}
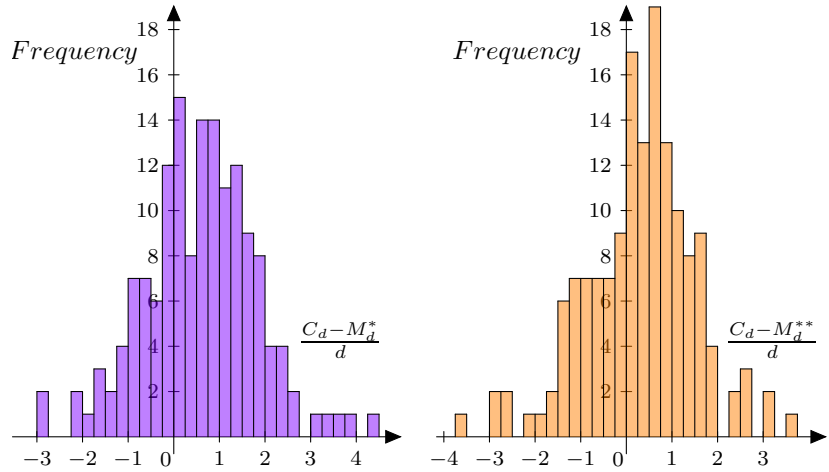

From the result of the first test in section \ref{ss2}, we would consider in this part that the function that maps $d$ onto $\frac{C_d-M^*_{d}}{d}$ 
behaves like a random variable with an expected value $\Lambda^*$ close to $0.576$ and with a symmetric probability distribution (for the considered 
values of $d$). On that assumption we will verify whether the values higher than $\Lambda^*$ and those smaller than $\Lambda^*$ are randomly scattered 
over the ordered absolute values of $\frac{C_d-A_d}{d}$ (the null hypothesis) or not. To this end we use a non-parametric test, the Mann–Whitney $U$ 
test: we determine the ranks of $|\frac{C_d-M^*_{d}}{d}|$ for each $d$ in the considered interval (see \cite{W} or \cite{MW}). The ranks sum of the 
values higher than $\Lambda^*$ is approximately normally distributed. The value of $U_1^*$ is about $-0.673$. The acceptance region of the hypothesis 
test with a $5\%$ risk being approximately $[-1.960,1.960]$, it can be concluded that we cannot reject the null hypothesis, \textit{i.e.} the fact that 
the greater and smaller than $\Lambda^*$ values of $\frac{C_d-M^*_{d}}{d}$ are randomly scattered: the symmetry of the probability distribution of this 
potential pseudorandom variable can not be rejected.

On the same assumption, we will also verify whether the values higher than $\Lambda^*$ and those smaller than $\Lambda^*$ are randomly scattered over 
the considered prime numbers (the null hypothesis) or not. To this end we again use the Mann–Whitney $U$ test. The prime numbers ranks sum of the values 
higher than $\Lambda^*$ is approximately normally distributed. The value of $U_2^*$ is about $0.721$. It can be concluded that we cannot reject the 
null hypothesis, \textit{i.e.} the fact that the greater and smaller than $\Lambda^*$ values of $\frac{C_d-M^*_{d}}{d}$ are randomly scattered over the 
considered prime numbers.

From the result of the second test in \ref{ss2}, we would consider in this part that the function that maps $d$ onto $\frac{C_d-M^{**}_{d}}{d}$ behaves 
like a random variable with an expected value $\Lambda^{**}$ close to $0.297$ and with a symmetric probability distribution (for the considered values 
of $d$). On that assumption we will verify whether the values higher than $\Lambda^{**}$ and those smaller than $\Lambda^{**}$ are randomly scattered 
over the ordered absolute values of $\frac{C_d-A_d}{d}$ (the null hypothesis) or not. To this end we again use the Mann–Whitney $U$ test. The value of 
$U_1^{**}$ is here about $-1.08$. It can once more be concluded that we cannot reject the fact that the greater and smaller than $\Lambda^{**}$ values 
of $\frac{C_d-M^{**}_{d}}{d}$ are randomly scattered: the symmetry of the probability distribution of this potential pseudorandom variable can not be 
rejected.

On the same assumption, we will verify whether the values higher than $\Lambda^{**}$ and those smaller than $\Lambda^{**}$ are randomly scattered over 
the considered prime numbers or not. To this end we again use the Mann–Whitney $U$ test. The value of $U_2^{**}$ is here about $-1.77$. It can once more 
be concluded that we cannot reject the fact that the greater and smaller than $\Lambda^{**}$ values of $\frac{C_d-M^{**}_{d}}{d}$ are randomly scattered 
over the considered prime numbers.

\subsubsection{Perspective}

Both first test and set of tests could not invalidate the fact that the function that maps $d$ onto $\frac{C_d-M^*_{d}}{d}$ and the one that maps $d$ 
onto $\frac{C_d-M^{**}_{d}}{d}$ seem to behave like random variables with respectively $\Lambda^*$ and $\Lambda^{**}$ as expected values. $\Lambda^*$ and 
$\Lambda^{**}$ are both positive numbers, whereas $\Lambda$ is negative; the added geometrical constraints seem to reduce in average the number of all 
but the simple points generated by a randomly chosen minimal Besicovitch arrangement. This reduction is slightly highter than expected. Our arrangements 
cannot obviously be limited to the considered geometrically constrained arrangement. Adding constraints for better modeling the arrangements and finding a way to 
determine whether the considered functions could be considered as high-quality pseudo-random number generators (PRNG) sketch some avenues for future 
research on the subject.

\newpage

\section*{Appendix}

\begin{table}[ht]
\tiny
\renewcommand{\arraystretch}{1.5}
\begin{tabular*}{1\textwidth}{@{\extracolsep{\fill}}|@{\hskip.6mm}c@{\hskip.6mm}||@{\hskip.6mm}c@{\hskip.6mm}|@{\hskip.6mm}c@{\hskip.6mm}|
@{\hskip.6mm}c@{\hskip.6mm}|@{\hskip.6mm}c@{\hskip.6mm}|@{\hskip.6mm}c@{\hskip.6mm}|@{\hskip.6mm}c@{\hskip.6mm}|@{\hskip.6mm}c@{\hskip.6mm}|
@{\hskip.6mm}c@{\hskip.6mm}|@{\hskip.6mm}c@{\hskip.6mm}|@{\hskip.6mm}c@{\hskip.6mm}|@{\hskip.6mm}c@{\hskip.6mm}|@{\hskip.6mm}c@{\hskip.6mm}|
@{\hskip.6mm}c@{\hskip.6mm}|@{\hskip.6mm}c@{\hskip.6mm}|@{\hskip.6mm}c@{\hskip.6mm}|@{\hskip.6mm}c@{\hskip.6mm}|@{\hskip.6mm}c@{\hskip.6mm}|
@{\hskip.6mm}c@{\hskip.6mm}|@{\hskip.6mm}c@{\hskip.6mm}|@{\hskip.6mm}c@{\hskip.6mm}|}
\hline
$d$ & $2$ & $3$ & $5^{**}$ & $7^{*}$ & $11^{**}$ & $13^{*}$ & $17^{**}$ & $19^{*}$ & $23^{**}$ & $29^{**}$ & $31^{*}$ & $37^{*}$ & $41^{**}$ & $43^{*}$ & $47^{**}$ & $53^{**}$ & $59$ & $61^{*}$ & $67^{*}$ & $71^{**}$ \\
\hline
$C_d$ & $1$ & $6$ & $13$ & $25$ & $67$ & $100$ & $163$ & $229$ & $334$ & $448$ & $625$ & $844$ & $1075$ & $1114$ & $1402$ & $1786$ & $1912$ & $2218$ & $2752$ & $3046$  \\
\hline
\end{tabular*}
\vspace{0.2\baselineskip}

\begin{tabular*}{1\textwidth}{@{\extracolsep{\fill}}|@{\hskip.6mm}c@{\hskip.6mm}||@{\hskip.6mm}c@{\hskip.6mm}|@{\hskip.6mm}c@{\hskip.6mm}|
@{\hskip.6mm}c@{\hskip.6mm}|@{\hskip.6mm}c@{\hskip.6mm}|@{\hskip.6mm}c@{\hskip.6mm}|@{\hskip.6mm}c@{\hskip.6mm}|@{\hskip.6mm}c@{\hskip.6mm}|
@{\hskip.6mm}c@{\hskip.6mm}|@{\hskip.6mm}c@{\hskip.6mm}|@{\hskip.6mm}c@{\hskip.6mm}|@{\hskip.6mm}c@{\hskip.6mm}|@{\hskip.6mm}c@{\hskip.6mm}|
@{\hskip.6mm}c@{\hskip.6mm}|@{\hskip.6mm}c@{\hskip.6mm}|@{\hskip.6mm}c@{\hskip.6mm}|@{\hskip.6mm}c@{\hskip.6mm}|}
\hline
$d$ & $73^{*}$ & $79$ & $83$ & $89^{**}$ & $97^{*}$ & $101^{**}$ & $103^{*}$ & $107^{**}$ & $109^{*}$ & $113^{**}$ & $127^{*}$ & $131^{**}$ & $137^{**}$ & $139^{*}$ & $149^{**}$ & $151^{*}$  \\
\hline
$C_d$ & $3307$ & $3685$ & $4189$ & $4972$ & $5971$ & $6367$ & $6475$ & $7102$ & $7315$ & $8107$ & $10150$ & $10879$ & $11824$ & $12220$ & $13936$ & $14176$  \\
\hline
\end{tabular*}
\vspace{0.2\baselineskip}

\begin{tabular*}{1\textwidth}{@{\extracolsep{\fill}}|@{\hskip.6mm}c@{\hskip.6mm}||@{\hskip.6mm}c@{\hskip.6mm}|@{\hskip.6mm}c@{\hskip.6mm}|
@{\hskip.6mm}c@{\hskip.6mm}|@{\hskip.6mm}c@{\hskip.6mm}|@{\hskip.6mm}c@{\hskip.6mm}|@{\hskip.6mm}c@{\hskip.6mm}|@{\hskip.6mm}c@{\hskip.6mm}|
@{\hskip.6mm}c@{\hskip.6mm}|@{\hskip.6mm}c@{\hskip.6mm}|@{\hskip.6mm}c@{\hskip.6mm}|@{\hskip.6mm}c@{\hskip.6mm}|@{\hskip.6mm}c@{\hskip.6mm}|
@{\hskip.6mm}c@{\hskip.6mm}|@{\hskip.6mm}c@{\hskip.6mm}|@{\hskip.6mm}c@{\hskip.6mm}|}
\hline
$d$ & $157^{*}$ & $163^{*}$ & $167^{**}$ & $173^{**}$ & $179$ & $181^{*}$ & $191^{**}$ & $193$ & $197^{**}$ & $199^{*}$ & $211^{*}$ & $223^{*}$ & $227$ & $229^{*}$ & $233^{**}$ \\
\hline
$C_d$ & $15529$ & $16546$ & $17440$ & $18799$ & $19789$ & $20758$ & $22945$ & $23251$ & $24430$ & $24739$ & $28186$ & $31348$ & $32482$ & $33127$ & $33721$ \\
\hline
\end{tabular*}
\vspace{0.2\baselineskip}

\begin{tabular*}{1\textwidth}{@{\extracolsep{\fill}}|@{\hskip.6mm}c@{\hskip.6mm}||@{\hskip.6mm}c@{\hskip.6mm}|@{\hskip.6mm}c@{\hskip.6mm}|
@{\hskip.6mm}c@{\hskip.6mm}|@{\hskip.6mm}c@{\hskip.6mm}|@{\hskip.6mm}c@{\hskip.6mm}|@{\hskip.6mm}c@{\hskip.6mm}|@{\hskip.6mm}c@{\hskip.6mm}|
@{\hskip.6mm}c@{\hskip.6mm}|@{\hskip.6mm}c@{\hskip.6mm}|@{\hskip.6mm}c@{\hskip.6mm}|@{\hskip.6mm}c@{\hskip.6mm}|@{\hskip.6mm}c@{\hskip.6mm}|
@{\hskip.6mm}c@{\hskip.6mm}|@{\hskip.6mm}c@{\hskip.6mm}|@{\hskip.6mm}c@{\hskip.6mm}|}
\hline
$d$ & $239^{**}$ & $241^{*}$ & $251^{**}$ & $257^{**}$ & $263^{**}$ & $269^{**}$ & $271^{*}$ & $277^{*}$ & $281^{**}$ & $283^{*}$ & $293^{**}$ & $307^{*}$ & $311^{**}$ & $313^{*}$ & $317^{**}$ \\
\hline
$C_d$ & $35800$ & $36577$ & $39808$ & $41515$ & $43795$ & $45214$ & $45940$ & $48160$ & $49507$ & $49747$ & $54625$ & $59248$ & $60886$ & $60592$ & $63535$  \\
\hline
\end{tabular*}
\vspace{0.2\baselineskip}

\begin{tabular*}{1\textwidth}{@{\extracolsep{\fill}}|@{\hskip.6mm}c@{\hskip.6mm}||@{\hskip.6mm}c@{\hskip.6mm}|@{\hskip.6mm}c@{\hskip.6mm}|
@{\hskip.6mm}c@{\hskip.6mm}|@{\hskip.6mm}c@{\hskip.6mm}|@{\hskip.6mm}c@{\hskip.6mm}|@{\hskip.6mm}c@{\hskip.6mm}|@{\hskip.6mm}c@{\hskip.6mm}|
@{\hskip.6mm}c@{\hskip.6mm}|@{\hskip.6mm}c@{\hskip.6mm}|@{\hskip.6mm}c@{\hskip.6mm}|@{\hskip.6mm}c@{\hskip.6mm}|@{\hskip.6mm}c@{\hskip.6mm}|
@{\hskip.6mm}c@{\hskip.6mm}|@{\hskip.6mm}c@{\hskip.6mm}|@{\hskip.6mm}c@{\hskip.6mm}|}
\hline
$d$ & $331^{*}$ & $337$ & $347^{**}$ & $349^{*}$ & $353^{**}$ & $359^{**}$ & $367^{*}$ & $373^{*}$ & $379^{*}$ & $383^{**}$ & $389^{**}$ & $397^{*}$ & $401^{**}$ & $409^{*}$ & $419$\\
\hline
$C_d$ & $68794$ & $71359$ & $74710$ & $76915$ & $78466$ & $81265$ & $84772$ & $87586$ & $90232$ & $92203$ & $95716$ & $99352$ & $101314$ & $104797$ & $109873$  \\
\hline
\end{tabular*}
\vspace{0.2\baselineskip}

\begin{tabular*}{1\textwidth}{@{\extracolsep{\fill}}|@{\hskip.6mm}c@{\hskip.6mm}||@{\hskip.6mm}c@{\hskip.6mm}|@{\hskip.6mm}c@{\hskip.6mm}|
@{\hskip.6mm}c@{\hskip.6mm}|@{\hskip.6mm}c@{\hskip.6mm}|@{\hskip.6mm}c@{\hskip.6mm}|@{\hskip.6mm}c@{\hskip.6mm}|@{\hskip.6mm}c@{\hskip.6mm}|
@{\hskip.6mm}c@{\hskip.6mm}|@{\hskip.6mm}c@{\hskip.6mm}|@{\hskip.6mm}c@{\hskip.6mm}|@{\hskip.6mm}c@{\hskip.6mm}|@{\hskip.6mm}c@{\hskip.6mm}|
@{\hskip.6mm}c@{\hskip.6mm}|}
\hline
$d$ & $421$ & $431^{**}$ & $433^{*}$ & $439^{*}$ & $443$ & $449^{**}$ & $457$ & $461^{**}$ & $463^{*}$ & $467^{**}$ & $479^{**}$ & $487^{*}$ & $491^{**}$  \\
\hline
$C_d$ & $111913$ & $117079$ & $118249$ & $122023$ & $123148$ & $127207$ & $130669$ & $133840$ & $134125$ & $136486$ & $144355$ & $150223$ & $151696$   \\
\hline
\end{tabular*}
\vspace{0.2\baselineskip}

\begin{tabular*}{1\textwidth}{@{\extracolsep{\fill}}|@{\hskip.6mm}c@{\hskip.6mm}||@{\hskip.6mm}c@{\hskip.6mm}|@{\hskip.6mm}c@{\hskip.6mm}|
@{\hskip.6mm}c@{\hskip.6mm}|@{\hskip.6mm}c@{\hskip.6mm}|@{\hskip.6mm}c@{\hskip.6mm}|@{\hskip.6mm}c@{\hskip.6mm}|@{\hskip.6mm}c@{\hskip.6mm}|
@{\hskip.6mm}c@{\hskip.6mm}|@{\hskip.6mm}c@{\hskip.6mm}|@{\hskip.6mm}c@{\hskip.6mm}|@{\hskip.6mm}c@{\hskip.6mm}|@{\hskip.6mm}c@{\hskip.6mm}|
@{\hskip.6mm}c@{\hskip.6mm}|}
\hline
$d$ & $499^{*}$ & $503^{**}$ & $509^{**}$ & $521^{**}$ & $523^{*}$ & $541^{*}$ & $547$ & $557^{**}$ & $563^{**}$ & $569^{**}$ & $571^{*}$ & $577^{*}$ & $587^{**}$ \\
\hline
$C_d$ & $157138$ & $159607$ & $162508$ & $171607$ & $172345$ & $183730$ & $188854$ & $195535$ & $200263$ & $204214$ & $203680$ & $210088$ & $216331$ \\
\hline
\end{tabular*}
\vspace{0.2\baselineskip}

\begin{tabular*}{1\textwidth}{@{\extracolsep{\fill}}|@{\hskip.6mm}c@{\hskip.6mm}||@{\hskip.6mm}c@{\hskip.6mm}|@{\hskip.6mm}c@{\hskip.6mm}|
@{\hskip.6mm}c@{\hskip.6mm}|@{\hskip.6mm}c@{\hskip.6mm}|@{\hskip.6mm}c@{\hskip.6mm}|@{\hskip.6mm}c@{\hskip.6mm}|@{\hskip.6mm}c@{\hskip.6mm}|
@{\hskip.6mm}c@{\hskip.6mm}|@{\hskip.6mm}c@{\hskip.6mm}|@{\hskip.6mm}c@{\hskip.6mm}|@{\hskip.6mm}c@{\hskip.6mm}|@{\hskip.6mm}c@{\hskip.6mm}|
@{\hskip.6mm}c@{\hskip.6mm}|}
\hline
$d$ & $593^{**}$ & $599^{**}$ & $601$ & $607^{*}$ & $613^{*}$ & $617^{**}$ & $619$ & $631^{*}$ & $641^{**}$ & $643^{*}$ & $647^{**}$ & $653^{**}$ & $659^{**}$ \\
\hline
$C_d$ & $221269$ & $226318$ & $227140$ & $232981$ & $237046$ & $239626$ & $242398$ & $250861$ & $259405$ & $260467$ & $263722$ & $268363$ & $273217$ \\
\hline
\end{tabular*}
\vspace{0.2\baselineskip}

\begin{tabular*}{1\textwidth}{@{\extracolsep{\fill}}|@{\hskip.6mm}c@{\hskip.6mm}||@{\hskip.6mm}c@{\hskip.6mm}|@{\hskip.6mm}c@{\hskip.6mm}|
@{\hskip.6mm}c@{\hskip.6mm}|@{\hskip.6mm}c@{\hskip.6mm}|@{\hskip.6mm}c@{\hskip.6mm}|@{\hskip.6mm}c@{\hskip.6mm}|@{\hskip.6mm}c@{\hskip.6mm}|
@{\hskip.6mm}c@{\hskip.6mm}|@{\hskip.6mm}c@{\hskip.6mm}|@{\hskip.6mm}c@{\hskip.6mm}|@{\hskip.6mm}c@{\hskip.6mm}|@{\hskip.6mm}c@{\hskip.6mm}|
@{\hskip.6mm}c@{\hskip.6mm}|}
\hline
$d$ & $661^{*}$ & $673^{*}$ & $677^{**}$ & $683^{**}$ & $691$ & $701$ & $709^{*}$ & $719^{**}$ & $727^{*}$ & $733^{*}$ & $739^{*}$ & $743^{**}$ & $751^{*}$ \\
\hline
$C_d$ & $275827$ & $286606$ & $288166$ & $294208$ & $299602$ & $312463$ & $319282$ & $325690$ & $332941$ & $338929$ & $344065$ & $347074$ & $353806$ \\
\hline
\end{tabular*}
\vspace{0.2\baselineskip}

\begin{tabular*}{1\textwidth}{@{\extracolsep{\fill}}|@{\hskip.6mm}c@{\hskip.6mm}||@{\hskip.6mm}c@{\hskip.6mm}|@{\hskip.6mm}c@{\hskip.6mm}|
@{\hskip.6mm}c@{\hskip.6mm}|@{\hskip.6mm}c@{\hskip.6mm}|@{\hskip.6mm}c@{\hskip.6mm}|@{\hskip.6mm}c@{\hskip.6mm}|@{\hskip.6mm}c@{\hskip.6mm}|
@{\hskip.6mm}c@{\hskip.6mm}|@{\hskip.6mm}c@{\hskip.6mm}|@{\hskip.6mm}c@{\hskip.6mm}|@{\hskip.6mm}c@{\hskip.6mm}|@{\hskip.6mm}c@{\hskip.6mm}|
@{\hskip.6mm}c@{\hskip.6mm}|}
\hline
$d$ & $757$ & $761^{**}$ & $769^{*}$ & $773^{**}$ & $787$ & $797^{**}$ & $809^{**}$ & $811^{*}$ & $821^{**}$ & $823^{*}$ & $827^{**}$ & $829^{*}$ & $839^{**}$ \\
\hline
$C_d$ & $360034$ & $364345$ & $373825$ & $377044$ & $390112$ & $400093$ & $413593$ & $416320$ & $424864$ & $425239$ & $431245$ & $436477$ & $443629$  \\
\hline
\end{tabular*}
\vspace{0.2\baselineskip}

\begin{tabular*}{1\textwidth}{@{\extracolsep{\fill}}|@{\hskip.6mm}c@{\hskip.6mm}||@{\hskip.6mm}c@{\hskip.6mm}|@{\hskip.6mm}c@{\hskip.6mm}|
@{\hskip.6mm}c@{\hskip.6mm}|@{\hskip.6mm}c@{\hskip.6mm}|@{\hskip.6mm}c@{\hskip.6mm}|@{\hskip.6mm}c@{\hskip.6mm}|@{\hskip.6mm}c@{\hskip.6mm}|
@{\hskip.6mm}c@{\hskip.6mm}|@{\hskip.6mm}c@{\hskip.6mm}|@{\hskip.6mm}c@{\hskip.6mm}|@{\hskip.6mm}c@{\hskip.6mm}|@{\hskip.6mm}c@{\hskip.6mm}|
@{\hskip.6mm}c@{\hskip.6mm}|}
\hline
$d$ & $853^{*}$ & $857$ & $859^{*}$ & $863^{**}$ & $877^{*}$ & $881^{**}$ & $883^{*}$ & $887$ & $907$ & $911$ & $919^{*}$ & $929$ & $937^{*}$  \\
\hline
$C_d$ & $458275$ & $463174$ & $466087$ & $472573$ & $483487$ & $488704$ & $491626$ & $494824$ & $519175$ & $523180$ & $533941$ & $543892$ & $553420$ \\
\hline
\end{tabular*}
\vspace{0.2\baselineskip}

\begin{tabular*}{1\textwidth}{@{\extracolsep{\fill}}|@{\hskip.6mm}c@{\hskip.6mm}||@{\hskip.6mm}c@{\hskip.6mm}|@{\hskip.6mm}c@{\hskip.6mm}|
@{\hskip.6mm}c@{\hskip.6mm}|@{\hskip.6mm}c@{\hskip.6mm}|@{\hskip.6mm}c@{\hskip.6mm}|@{\hskip.6mm}c@{\hskip.6mm}|@{\hskip.6mm}c@{\hskip.6mm}|
@{\hskip.6mm}c@{\hskip.6mm}|@{\hskip.6mm}c@{\hskip.6mm}|@{\hskip.6mm}c@{\hskip.6mm}|@{\hskip.6mm}c@{\hskip.6mm}|@{\hskip.6mm}c@{\hskip.6mm}|
@{\hskip.6mm}c@{\hskip.6mm}|}
\hline
$d$ & $941^{**}$ & $947^{**}$ & $953^{**}$ & $967^{*}$ & $971$ & $977$ & $983^{**}$ & $991^{*}$ & $997^{*}$ & $1009^{*}$ & $1013^{**}$ & $1019^{**}$ & $1021^{*}$ \\
\hline
$C_d$ & $559363$ & $565651$ & $574390$ & $589471$ & $594424$ & $599923$ & $610498$ & $620311$ & $627001$ & $644440$ & $644356$ & $653449$ & $658795$ \\
\hline
\end{tabular*}
\vspace{0.2\baselineskip}

\begin{tabular*}{1\textwidth}{@{\extracolsep{\fill}}|@{\hskip.6mm}c@{\hskip.6mm}||@{\hskip.6mm}c@{\hskip.6mm}|@{\hskip.6mm}c@{\hskip.6mm}|
@{\hskip.6mm}c@{\hskip.6mm}|@{\hskip.6mm}c@{\hskip.6mm}|@{\hskip.6mm}c@{\hskip.6mm}|@{\hskip.6mm}c@{\hskip.6mm}|@{\hskip.6mm}c@{\hskip.6mm}|
@{\hskip.6mm}c@{\hskip.6mm}|@{\hskip.6mm}c@{\hskip.6mm}|@{\hskip.6mm}c@{\hskip.6mm}|@{\hskip.6mm}c@{\hskip.6mm}|@{\hskip.6mm}c@{\hskip.6mm}|
@{\hskip.6mm}c@{\hskip.6mm}|}
\hline
$d$ & $1031^{**}$ & $1033^{*}$ & $1039$ & $1049^{**}$ & $1051^{*}$ & $1061^{**}$ & $1063^{*}$ & $1069^{*}$ & $1087^{*}$ & $1091$ & $1093$ & $1097^{**}$ & $1103^{**}$  \\
\hline
$C_d$ & $671311$& $674257$ & $680911$ & $692635$ & $697756$ & $710902$ & $712840$ & $723076$ & $745966$ & $752482$ & $752740$ & $759808$ & $768805$ \\
\hline
\end{tabular*}
\vspace{0.2\baselineskip}

\begin{tabular*}{1\textwidth}{@{\extracolsep{\fill}}|@{\hskip.6mm}c@{\hskip.6mm}||@{\hskip.6mm}c@{\hskip.6mm}|@{\hskip.6mm}c@{\hskip.6mm}|
@{\hskip.6mm}c@{\hskip.6mm}|@{\hskip.6mm}c@{\hskip.6mm}|@{\hskip.6mm}c@{\hskip.6mm}|@{\hskip.6mm}c@{\hskip.6mm}|@{\hskip.6mm}c@{\hskip.6mm}|
@{\hskip.6mm}c@{\hskip.6mm}|@{\hskip.6mm}c@{\hskip.6mm}|@{\hskip.6mm}c@{\hskip.6mm}|@{\hskip.6mm}c@{\hskip.6mm}|@{\hskip.6mm}c@{\hskip.6mm}|
@{\hskip.6mm}c@{\hskip.6mm}|}
\hline
$d$ & $1109$ & $1117^{*}$ & $1123^{*}$ & $1129^{*}$ & $1151^{**}$ & $1153^{*}$ & $1163^{**}$ & $1171^{*}$ & $1181^{**}$ & $1187^{**}$ & $1193$ & $1201^{*}$ & $1213^{*}$   \\
\hline
$C_d$ & $779941$ & $787798$ & $794254$ & $806077$ & $837823$ & $838891$ & $851632$ & $862882$ & $878656$ & $887017$ & $900982$ & $911497$ & $929935$ \\
\hline
\end{tabular*}
\vspace{0.2\baselineskip}

\begin{tabular*}{1\textwidth}{@{\extracolsep{\fill}}|@{\hskip.6mm}c@{\hskip.6mm}||@{\hskip.6mm}c@{\hskip.6mm}|@{\hskip.6mm}c@{\hskip.6mm}|
@{\hskip.6mm}c@{\hskip.6mm}|@{\hskip.6mm}c@{\hskip.6mm}|@{\hskip.6mm}c@{\hskip.6mm}|@{\hskip.6mm}c@{\hskip.6mm}|@{\hskip.6mm}c@{\hskip.6mm}|
@{\hskip.6mm}c@{\hskip.6mm}|@{\hskip.6mm}c@{\hskip.6mm}|@{\hskip.6mm}c@{\hskip.6mm}|@{\hskip.6mm}c@{\hskip.6mm}|@{\hskip.6mm}c@{\hskip.6mm}|}
\hline
$d$ & $1217$ & $1223$ & $1229^{**}$ & $1231$ & $1237$ & $1249^{*}$ & $1259$ & $1277^{**}$ & $1279^{*}$ & $1283$ & $1289$ & $1291^{*}$  \\
\hline
$C_d$ & $936253$ & $943267$ & $956872$ & $956560$ & $964465$ & $985237$ & $1000621$ & $1029562$ & $1033756$ & $1039588$ & $1047226$ & $1052251$ \\
\hline
\end{tabular*}
\vspace{0.2\baselineskip}

\begin{tabular*}{1\textwidth}{@{\extracolsep{\fill}}|@{\hskip.6mm}c@{\hskip.6mm}||@{\hskip.6mm}c@{\hskip.6mm}|@{\hskip.6mm}c@{\hskip.6mm}|
@{\hskip.6mm}c@{\hskip.6mm}|@{\hskip.6mm}c@{\hskip.6mm}|@{\hskip.6mm}c@{\hskip.6mm}|@{\hskip.6mm}c@{\hskip.6mm}|@{\hskip.6mm}c@{\hskip.6mm}|
@{\hskip.6mm}c@{\hskip.6mm}|@{\hskip.6mm}c@{\hskip.6mm}|@{\hskip.6mm}c@{\hskip.6mm}|@{\hskip.6mm}c@{\hskip.6mm}|@{\hskip.6mm}c@{\hskip.6mm}|}
\hline
$d$ & $1297^{*}$ & $1301^{**}$ & $1303$ & $1307^{**}$ & $1319^{**}$ & $1321^{*}$ & $1327^{*}$ & $1361^{**}$ & $1367^{**}$ & $1373^{**}$ & $1381^{*}$ & $1399^{*}$  \\
\hline
$C_d$ & $1063438$ & $1068115$ & $1071913$ & $1078375$ & $1101274$ & $1102360$ & $1113577$ & $1169632$ & $1181578$ & $1192081$ & $1205287$ & $1235425$ \\
\hline
\end{tabular*}
\vspace{0.2\baselineskip}

\begin{tabular*}{1\textwidth}{@{\extracolsep{\fill}}|@{\hskip.6mm}c@{\hskip.6mm}||@{\hskip.6mm}c@{\hskip.6mm}|@{\hskip.6mm}c@{\hskip.6mm}|
@{\hskip.6mm}c@{\hskip.6mm}|@{\hskip.6mm}c@{\hskip.6mm}|@{\hskip.6mm}c@{\hskip.6mm}|@{\hskip.6mm}c@{\hskip.6mm}|@{\hskip.6mm}c@{\hskip.6mm}|
@{\hskip.6mm}c@{\hskip.6mm}|@{\hskip.6mm}c@{\hskip.6mm}|@{\hskip.6mm}c@{\hskip.6mm}|@{\hskip.6mm}c@{\hskip.6mm}|@{\hskip.6mm}c@{\hskip.6mm}|}
\hline
$d$ & $1409^{**}$ & $1423^{*}$ & $1427^{**}$ & $1429^{*}$ & $1433^{**}$ & $1439$ & $1447^{*}$ & $1451^{**}$ & $1453^{*}$ & $1459^{*}$ & $1471^{*}$ & $1481^{**}$   \\
\hline
$C_d$ & $1251661$ & $1280677$ & $1281337$ & $1291003$ & $1294351$ & $1303495$ & $1326448$ & $1329037$ & $1330435$ & $1344493$ & $1364623$ & $1385854$ \\
\hline
\end{tabular*}
\vspace{0.2\baselineskip}

\begin{tabular*}{1\textwidth}{@{\extracolsep{\fill}}|@{\hskip.6mm}c@{\hskip.6mm}||@{\hskip.6mm}c@{\hskip.6mm}|@{\hskip.6mm}c@{\hskip.6mm}|
@{\hskip.6mm}c@{\hskip.6mm}|@{\hskip.6mm}c@{\hskip.6mm}|@{\hskip.6mm}c@{\hskip.6mm}|@{\hskip.6mm}c@{\hskip.6mm}|@{\hskip.6mm}c@{\hskip.6mm}|
@{\hskip.6mm}c@{\hskip.6mm}|@{\hskip.6mm}c@{\hskip.6mm}|@{\hskip.6mm}c@{\hskip.6mm}|@{\hskip.6mm}c@{\hskip.6mm}|@{\hskip.6mm}c@{\hskip.6mm}|}
\hline
$d$ & $1483^{*}$ & $1487$ & $1489$ & $1493$ & $1499^{**}$ & $1511^{**}$ & $1523^{**}$ & $1531$ & $1543^{*}$ & $1549^{*}$ & $1553^{**}$ & $1559^{**}$  \\
\hline
$C_d$ & $1387588$ & $1398910$ & $1400191$ & $1407457$ & $1420246$ & $1444459$ & $1464190$ & $1477492$ & $1502308$ & $1513120$ & $1524895$ & $1533844$ \\
\hline
\end{tabular*}
\vspace{0.2\baselineskip}

\begin{tabular*}{1\textwidth}{@{\extracolsep{\fill}}|@{\hskip.6mm}c@{\hskip.6mm}||@{\hskip.6mm}c@{\hskip.6mm}|@{\hskip.6mm}c@{\hskip.6mm}|
@{\hskip.6mm}c@{\hskip.6mm}|@{\hskip.6mm}c@{\hskip.6mm}|@{\hskip.6mm}c@{\hskip.6mm}|@{\hskip.6mm}c@{\hskip.6mm}|@{\hskip.6mm}c@{\hskip.6mm}|
@{\hskip.6mm}c@{\hskip.6mm}|@{\hskip.6mm}c@{\hskip.6mm}|@{\hskip.6mm}c@{\hskip.6mm}|@{\hskip.6mm}c@{\hskip.6mm}|@{\hskip.6mm}c@{\hskip.6mm}|}
\hline
$d$ & $1567^{*}$ & $1571^{**}$ & $1579^{*}$ & $1583^{**}$ & $1597^{*}$ & $1601^{**}$ & $1607^{**}$ & $1609^{*}$ & $1613$ & $1619^{**}$ & $1621^{*}$ & $1627^{*}$     \\
\hline
$C_d$ & $1549756$ & $1558054$ & $1571542$ & $1584523$ & $1609036$ & $1614859$ & $1630015$ & $1638106$ & $1644892$ & $1655251$ & $1659781$ & $1673800$ \\
\hline
\end{tabular*}
\vspace{0.2\baselineskip}

\begin{tabular*}{1\textwidth}{@{\extracolsep{\fill}}|@{\hskip.6mm}c@{\hskip.6mm}||@{\hskip.6mm}c@{\hskip.6mm}|@{\hskip.6mm}c@{\hskip.6mm}|
@{\hskip.6mm}c@{\hskip.6mm}|@{\hskip.6mm}c@{\hskip.6mm}|@{\hskip.6mm}c@{\hskip.6mm}|@{\hskip.6mm}c@{\hskip.6mm}|@{\hskip.6mm}c@{\hskip.6mm}|
@{\hskip.6mm}c@{\hskip.6mm}|@{\hskip.6mm}c@{\hskip.6mm}|@{\hskip.6mm}c@{\hskip.6mm}|@{\hskip.6mm}c@{\hskip.6mm}|@{\hskip.6mm}c@{\hskip.6mm}|}
\hline
$d$ & $1637$ & $1657$ & $1663^{*}$ & $1667^{**}$ & $1669^{*}$ & $1693^{*}$ & $1697^{**}$ & $1699^{*}$ & $1709^{**}$ & $1721^{**}$ & $1723^{*}$ & $1733^{**}$   \\
\hline
$C_d$ & $1692076$ & $1735675$ & $1746766$ & $1755874$ & $1759345$ & $1811281$ & $1817827$ & $1821148$ & $1845148$ & $1868239$ & $1875610$ & $1893445$ \\
\hline
\end{tabular*}
\end{table}

\begin{table}[ht]
\tiny
\renewcommand{\arraystretch}{1.5}
\begin{tabular*}{1\textwidth}{@{\extracolsep{\fill}}|@{\hskip.6mm}c@{\hskip.6mm}||@{\hskip.6mm}c@{\hskip.6mm}|@{\hskip.6mm}c@{\hskip.6mm}|
@{\hskip.6mm}c@{\hskip.6mm}|@{\hskip.6mm}c@{\hskip.6mm}|@{\hskip.6mm}c@{\hskip.6mm}|@{\hskip.6mm}c@{\hskip.6mm}|@{\hskip.6mm}c@{\hskip.6mm}|
@{\hskip.6mm}c@{\hskip.6mm}|@{\hskip.6mm}c@{\hskip.6mm}|@{\hskip.6mm}c@{\hskip.6mm}|@{\hskip.6mm}c@{\hskip.6mm}|@{\hskip.6mm}c@{\hskip.6mm}|}
\hline
$d$ & $1741^{*}$ & $1747^{*}$ & $1753^{*}$ & $1759^{*}$ & $1777^{*}$ & $1783^{*}$ & $1787^{**}$ & $1789^{*}$ & $1801^{*}$ & $1811$ & $1823^{**}$ & $1831^{*}$  \\
\hline
$C_d$ & $1915426$ & $1926787$ & $1938808$ & $1956460$ & $1991359$ & $2006128$ & $2018023$ & $2019100$ & $2055712$ & $2074435$ & $2092648$ & $2118334$ \\
\hline
\end{tabular*}
\vspace{0.2\baselineskip}

\begin{tabular*}{1\textwidth}{@{\extracolsep{\fill}}|@{\hskip.6mm}c@{\hskip.6mm}||@{\hskip.6mm}c@{\hskip.6mm}|@{\hskip.6mm}c@{\hskip.6mm}|
@{\hskip.6mm}c@{\hskip.6mm}|@{\hskip.6mm}c@{\hskip.6mm}|@{\hskip.6mm}c@{\hskip.6mm}|@{\hskip.6mm}c@{\hskip.6mm}|@{\hskip.6mm}c@{\hskip.6mm}|
@{\hskip.6mm}c@{\hskip.6mm}|@{\hskip.6mm}c@{\hskip.6mm}|@{\hskip.6mm}c@{\hskip.6mm}|@{\hskip.6mm}c@{\hskip.6mm}|@{\hskip.6mm}c@{\hskip.6mm}|}
\hline
$d$ & $1847$ & $1861^{*}$ & $1867^{*}$ & $1871^{**}$ & $1873^{*}$ & $1877^{**}$ & $1879^{*}$ & $1889^{**}$ & $1901$ & $1907^{**}$ & $1913^{**}$ & $1931^{**}$ \\
\hline
$C_d$ & $2156626$ & $2185264$ & $2198473$ & $2211484$ & $2216392$ & $2224747$ & $2228053$ & $2253946$ & $2281783$ & $2297935$ & $2303611$ & $2355019$ \\
\hline
\end{tabular*}
\vspace{0.2\baselineskip}

\begin{tabular*}{1\textwidth}{@{\extracolsep{\fill}}|@{\hskip.6mm}c@{\hskip.6mm}||@{\hskip.6mm}c@{\hskip.6mm}|@{\hskip.6mm}c@{\hskip.6mm}|
@{\hskip.6mm}c@{\hskip.6mm}|@{\hskip.6mm}c@{\hskip.6mm}|@{\hskip.6mm}c@{\hskip.6mm}|@{\hskip.6mm}c@{\hskip.6mm}|@{\hskip.6mm}c@{\hskip.6mm}|
@{\hskip.6mm}c@{\hskip.6mm}|@{\hskip.6mm}c@{\hskip.6mm}|@{\hskip.6mm}c@{\hskip.6mm}|@{\hskip.6mm}c@{\hskip.6mm}|@{\hskip.6mm}c@{\hskip.6mm}|}
\hline
$d$ & $1933^{*}$ & $1949^{**}$ & $1951^{*}$ & $1973^{**}$ & $1979^{**}$ & $1987^{*}$ & $1993$ & $1997$ & $1999^{*}$ & $2003$ & $2011^{*}$ & $2017^{*}$ \\
\hline
$C_d$ & $2356819$ & $2398531$ & $2407693$ & $2459041$ & $2474182$ & $2493151$ & $2513734$ & $2520214$ & $2525929$ & $2534818$ & $2554063$ & $2571514$ \\
\hline
\end{tabular*}
\vspace{0.2\baselineskip}

\begin{tabular*}{1\textwidth}{@{\extracolsep{\fill}}|@{\hskip.6mm}c@{\hskip.6mm}||@{\hskip.6mm}c@{\hskip.6mm}|@{\hskip.6mm}c@{\hskip.6mm}|
@{\hskip.6mm}c@{\hskip.6mm}|@{\hskip.6mm}c@{\hskip.6mm}|@{\hskip.6mm}c@{\hskip.6mm}|@{\hskip.6mm}c@{\hskip.6mm}|@{\hskip.6mm}c@{\hskip.6mm}|
@{\hskip.6mm}c@{\hskip.6mm}|@{\hskip.6mm}c@{\hskip.6mm}|@{\hskip.6mm}c@{\hskip.6mm}|@{\hskip.6mm}c@{\hskip.6mm}|@{\hskip.6mm}c@{\hskip.6mm}|}
\hline
$d$ & $2027^{**}$ & $2029^{*}$ & $2039^{**}$ & $2053^{*}$ & $2063^{**}$ & $2069^{**}$ & $2081^{**}$ & $2083^{*}$ & $2087$ & $2089$ & $2099^{**}$ & $2111^{**}$ \\
\hline
$C_d$ & $2594968$ & $2605618$ & $2625871$ & $2661322$ & $2685235$ & $2700313$ & $2739367$ & $2741827$ & $2750443$ & $2757349$ & $2783251$ & $2809579$ \\
\hline
\end{tabular*}
\vspace{0.2\baselineskip}

\begin{tabular*}{1\textwidth}{@{\extracolsep{\fill}}|@{\hskip.6mm}c@{\hskip.6mm}||@{\hskip.6mm}c@{\hskip.6mm}|@{\hskip.6mm}c@{\hskip.6mm}|
@{\hskip.6mm}c@{\hskip.6mm}|@{\hskip.6mm}c@{\hskip.6mm}|@{\hskip.6mm}c@{\hskip.6mm}|@{\hskip.6mm}c@{\hskip.6mm}|@{\hskip.6mm}c@{\hskip.6mm}|
@{\hskip.6mm}c@{\hskip.6mm}|@{\hskip.6mm}c@{\hskip.6mm}|@{\hskip.6mm}c@{\hskip.6mm}|@{\hskip.6mm}c@{\hskip.6mm}|@{\hskip.6mm}c@{\hskip.6mm}|}
\hline
$d$ & $2113$ & $2129^{**}$ & $2131^{*}$ & $2137^{*}$ & $2141^{**}$ & $2143^{*}$ & $2153^{**}$ & $2161^{*}$ & $2179^{*}$ & $2203^{*}$ & $2207^{**}$ & $2213^{**}$ \\
\hline
$C_d$ & $2821639$ & $2862883$ & $2869180$ & $2886973$ & $2892547$ & $2898181$ & $2928235$ & $2952049$ & $3001276$ & $3069025$ & $3075811$ & $3092218$ \\
\hline
\end{tabular*}
\vspace{0.2\baselineskip}

\begin{tabular*}{1\textwidth}{@{\extracolsep{\fill}}|@{\hskip.6mm}c@{\hskip.6mm}||@{\hskip.6mm}c@{\hskip.6mm}|@{\hskip.6mm}c@{\hskip.6mm}|
@{\hskip.6mm}c@{\hskip.6mm}|@{\hskip.6mm}c@{\hskip.6mm}|@{\hskip.6mm}c@{\hskip.6mm}|@{\hskip.6mm}c@{\hskip.6mm}|@{\hskip.6mm}c@{\hskip.6mm}|
@{\hskip.6mm}c@{\hskip.6mm}|@{\hskip.6mm}c@{\hskip.6mm}|@{\hskip.6mm}c@{\hskip.6mm}|@{\hskip.6mm}c@{\hskip.6mm}|@{\hskip.6mm}c@{\hskip.6mm}|}
\hline
$d$ & $2221^{*}$ & $2237^{**}$ & $2239^{*}$ & $2243$ & $2251$ & $2267^{**}$ & $2269^{*}$ & $2273^{**}$ & $2281^{*}$ & $2287^{*}$ & $2293^{*}$ & $2297^{**}$ \\
\hline
$C_d$ & $3114424$ & $3159310$ & $3166807$ & $3175720$ & $3199828$ & $3244723$ & $3256783$ & $3265912$ & $3285589$ & $3303373$ & $3326029$ & $3330658$ \\
\hline
\end{tabular*}
\vspace{0.2\baselineskip}

\begin{tabular*}{1\textwidth}{@{\extracolsep{\fill}}|@{\hskip.6mm}c@{\hskip.6mm}||@{\hskip.6mm}c@{\hskip.6mm}|@{\hskip.6mm}c@{\hskip.6mm}|
@{\hskip.6mm}c@{\hskip.6mm}|@{\hskip.6mm}c@{\hskip.6mm}|@{\hskip.6mm}c@{\hskip.6mm}|@{\hskip.6mm}c@{\hskip.6mm}|@{\hskip.6mm}c@{\hskip.6mm}|
@{\hskip.6mm}c@{\hskip.6mm}|@{\hskip.6mm}c@{\hskip.6mm}|@{\hskip.6mm}c@{\hskip.6mm}|@{\hskip.6mm}c@{\hskip.6mm}|@{\hskip.6mm}c@{\hskip.6mm}|}
\hline
$d$ & $2309^{**}$ & $2311$ & $2333^{**}$ & $2339^{**}$ & $2341^{*}$ & $2347^{*}$ & $2351^{**}$ & $2357^{**}$ & $2371^{*}$ & $2377$ & $2381^{**}$ & $2383^{*}$ \\
\hline
$C_d$ & $3372679$ & $3373075$ & $3434839$ & $3457402$ & $3462010$ & $3480025$ & $3497599$ & $3510505$ & $3555751$ & $3567400$ & $3579163$ & $3587740$ \\
\hline
\end{tabular*}
\vspace{0.2\baselineskip}

\begin{tabular*}{1\textwidth}{@{\extracolsep{\fill}}|@{\hskip.6mm}c@{\hskip.6mm}||@{\hskip.6mm}c@{\hskip.6mm}|@{\hskip.6mm}c@{\hskip.6mm}|
@{\hskip.6mm}c@{\hskip.6mm}|@{\hskip.6mm}c@{\hskip.6mm}|@{\hskip.6mm}c@{\hskip.6mm}|@{\hskip.6mm}c@{\hskip.6mm}|@{\hskip.6mm}c@{\hskip.6mm}|
@{\hskip.6mm}c@{\hskip.6mm}|@{\hskip.6mm}c@{\hskip.6mm}|@{\hskip.6mm}c@{\hskip.6mm}|@{\hskip.6mm}c@{\hskip.6mm}|@{\hskip.6mm}c@{\hskip.6mm}|}
\hline
$d$ & $2389$ & $2393^{**}$ & $2399^{**}$ & $2411^{**}$ & $2417^{**}$ & $2423$ & $2437$ & $2441^{**}$ & $2447^{**}$ & $2459^{**}$ & $2467^{*}$ & $2473^{*}$ \\
\hline
$C_d$ & $3602248$ & $3614269$ & $3636025$ & $3671155$ & $3687757$ & $3709300$ & $3749929$ & $3761812$ & $3780007$ & $3821119$ & $3847576$ & $3861457$ \\
\hline
\end{tabular*}
\vspace{0.2\baselineskip}

\begin{tabular*}{0.36\textwidth}{@{\extracolsep{\fill}}|@{\hskip.6mm}c@{\hskip.6mm}||@{\hskip.6mm}c@{\hskip.6mm}|@{\hskip.6mm}c@{\hskip.6mm}|
@{\hskip.6mm}c@{\hskip.6mm}|@{\hskip.6mm}c@{\hskip.6mm}|}
\hline
$d$ & $2477$ & $2503^{*}$ & $2521^{*}$ & $2531^{**}$ \\
\hline
$C_d$ & $3878440$ & $3960268$ & $4014841$ & $4046863$ \\
\hline
\end{tabular*}

\caption{The complexities values. Values of $d$ with one asterisk correspond to arrangements where $d\equiv1\mod3$ and where all the lines (except those 
of $\{L_0,L_{-1},L_\infty\}$ and $\{L_\omega,L_{\omega^2}\}$) do not pass through the origin, whereas values of $d$ with two asterisks correspond to 
arrangements where $d\equiv2\mod3$ and where all the lines (except those of $\{L_0,L_{-1},L_\infty\}$) do not pass through the origin.}
\label{Comp}
\end{table}

\newpage

\bibliography{bib}

\end{document}